\documentclass{article}

\usepackage[
backend=biber,
sorting=nyt
]{biblatex}
\renewbibmacro{in:}{}

\usepackage{graphicx} 
\usepackage{mathrsfs}
\usepackage[a4paper, total={6in, 9in}]{geometry}
\usepackage{mathtools}
\usepackage{enumitem}
\usepackage{amsmath}
\usepackage{amsthm}
\usepackage{amssymb}
\usepackage{graphicx}
\usepackage[dvipsnames]{xcolor}

\newtheorem{theorem}{Theorem}[subsection]
\newtheorem{corollary}[theorem]{Corollary}
\newtheorem{lemma}[theorem]{Lemma}
\newtheorem{proposition}[theorem]{Proposition}

\newtheorem{remark}[theorem]{Remark}

\usepackage[dvipsnames]{xcolor}
\usepackage[colorlinks=true]{hyperref}
\hypersetup{urlcolor=RubineRed,linkcolor=RoyalBlue,citecolor=ForestGreen}

\usepackage{sectsty}
\sectionfont{\centering}
\numberwithin{equation}{section}

\title{\large THE WEIGHT TWO AND OPPOSITE SIGN CASES FOR THE FOURIER RELATIVE TRACE FORMULAS \normalsize}
\author{Matteo Di Scipio}
\date{April 2026}

\begin{document}

\maketitle
\begin{abstract} 
\noindent We provide an adelic relative trace formula proof to the Petersson/Bruggeman-Kuznetsov (PBK) formulas, specifically in the holomorphic case for $\kappa=2$ and the non-holomorphic case for $m_1m_2<0$. Given two sets of hypothesis on the non archimedean test function $f$, called the geometric and spectral assumptions, this approach allows us to obtain refined PBK formulas.
\end{abstract}
\tableofcontents
\pagebreak 

\section{Introduction}
\subsection{Motivation}

The Petersson and the Bruggeman-Kuznetsov formulas are some of the most important tools in analytic number theory. Applications abound in literature, see for instance \cite{Conrey-Iwaniec2000}, \cite{Blomer-Harcos2008}, \cite{PetrowYoung2020}, \cite{Hu-Petrow-Young-moment} for moments and subconvexity, \cite{Deshouillers} for large sieve inequalities, \cite{Duke1995} and \cite{Iwaniec-Sarnak2000} for non-vanishing of L-functions at central values, \cite{Bombieri1986} and \cite{maynard2021} for primes in arithmetic progression, \cite{Iwaniec2000} for low lying zeroes, and \cite{Michel2006Kloostermania} for a survey on Kloostermania and its applications. \bigbreak 
There are two main ways of obtaining these formulas: a more classical approach involving Poincaré series (see e.g \cite{Petersson2} and \cite{Kuznetsov}) and the pre-trace formula approach developed by Jacquet \cite{Jacquet1986} and Zagier \cite{Zagier1981}, as exposited by Knigthly and Li in \cite{KL-petersson} and \cite{KL-kuznetsov} for the  Petersson formula of weight $\kappa>2$ and the Bruggeman-Kuznetsov formula for pairs of coefficient with the same sign. 
One of the advantages of working with the pre-trace formula is the greater flexibility in selecting families of local non-archimedean representations underlying automorphic forms, obtaining generalised formulas as in \cite{generalised}. The aim of this paper is to fill a gap in the literature, developing similar formulas for Maass forms in the opposite sign case, as well as in the $\kappa = 2$ holomorphic case via the pre-trace approach. It is worth highligthing that the holomorphic case is particularly interesting for applications given the connections with elliptic curves over $\mathbb{Q}$. \bigbreak

We first prove a version of the opposite sign formula under minimal assumptions; the geometric side will be in terms of generalised Kloosterman sums, as in \cite{generalised}. Following the same paper, we outline how additional constraints of the test function, namely the \emph{spectral} and \emph{geometric assumptions}, allow to obtain refined formulas. We prove this for a large class of archimedean test functions $h_\infty$, namely those that satisfy
\begin{equation}\label{def:testfunction-conditions}
    \begin{cases}
        h_\infty(t) \text{ is even} \\
        h_\infty(t) \ \text{is holomorphic in } |\text{Im}(t)|\leq A \\
        h_\infty(t)\ll (1+|t|)^{-B} \ \text{in } |\text{Im}(t)|\leq A
    \end{cases}
\end{equation}
for any $A>13$ and $B>2$ as in \cite{KL-kuznetsov} (see Theorem 8.1 and its proof loc. cit.). The derivation of this formula runs very similar to the same sign case, with the only major change arising on the archimedean component of the off-diagonal terms on the geometric side, for which we introduce a counterpart to the Zagier transform given in \cite[\S 7.5.2]{KL-kuznetsov}.  
As a simple corollary, we prove an equidistribution result for the parities of automorphic representations that are unramified at infinity with specified non-archimedean components.

We also prove the refined weight two Petersson formula by a limiting argument, taking a smooth truncation of the matrix coefficient associated with the newform of the discrete series $\mathcal{D}(2)$. We approach the problem in this way because, unlike in the case $\kappa \geq 4$ described in \cite{KL-petersson}, the matrix coefficient is not in $L^1([G])$, and the geometric kernel attached to it need not be absolutely and uniformly convergent on compacta. It is worth to also point out this method differs from the matrix coefficient approach used for Selberg trace formulas (see for instance \cite[Corollary 7.5.4]{palm2012explicitgl2traceformulas}).

We mention in passing that we chose to work over $\mathbb{Q}$ for sake of presentation, but the results of this paper should generalise to arbitrary number fields. The estimates we present throughout in order to prove the main theorems are relatively weak, and are not the best ones known to date; this is done in the hope that similar techniques may find application in similar problems.

\subsection{Statement of main results}
We refer the reader to Section 1.4.2 for notation concerning groups, which we tried to keep as standard as possible.
For a choice of a given test function $f_\mathbb{A}=f_\infty \otimes f\in C(G(\mathbb{A}))$,
the trace formulas we shall discuss provide an equality between two series. One side, commonly referred to as the spectral side, involves a product of pairs of Fourier coefficients, averaged over a family of automorphic forms cut by $f$, whereas the other involved generalised Kloosterman sums. Choosing $f_\infty \in C_c^\infty(G^+(\mathbb{R}))$ which is also bi-$K_\infty$-invariant, and $f\in C_c^\infty(G(\mathbb{A}_{\text{fin}}))$ with mild support conditions, one may obtain an unrefined version of the  Bruggeman-Kuznetsov formula. An interesting feature of such a formula is that one may suppress $f_\infty$ entirely, relating instead an entire function $h_\infty$ on the spectral side to its integral transform $H_\infty$ on the geometric side (see Proposition \ref{prop:paley-wiener-isom} for the relation between $f_\infty$ and $h_\infty$).

\begin{theorem}[Unrefined BK formula]\label{thm:unrefinedBKformula} Suppose $f=\bigotimes_p f_p$ is a non-zero, locally constant and compactly supported function on $G(\mathbb{A}_{\operatorname{fin}})$, and that for each $p$ the local test function $f_p$ is supported inside the subgroup of matrices $g \in G(\mathbb{Q}_p)$ with $v_p(\det g)\in 2 \mathbb{Z}$.\\
For $m_1,m_2 \in \tfrac{1}{N}\mathbb{Z}$ with $m_1m_2<0$ we have that 
\begin{align}
    \sum_{\pi \in \mathcal{F}_0(f)} & h_\infty(t_\pi)\sum_{\varphi \in \mathcal{B}(\pi)}a_{u_{\pi(f)\varphi}}(m_1)\overline{a_{u_\varphi}(m_2)} \notag \\
    &+\frac{1}{4\pi}\sum_{\chi \in \mathcal{F}_E(f)}\sum_{\phi\in \mathcal{B}(\pi_{\chi,\chi^{-1}})}\int_{-\infty}^{\infty}h_\infty(t)a_{u_{E(\pi_{it}(\chi)(f)\phi_{it})}}(m_1)\overline{a_{u_{E(\phi_{it})}}(m_2)}dt \\
    & = \sum_{c \in \mathcal{C}(f)}\frac{H(m_1,m_2,c)}{c}H_\infty^-\left(\frac{4\pi\sqrt{|m_1m_2|}}{c} \right) \notag
\end{align}
as absolute convergent sums/integrals. Here
    \begin{itemize}[itemsep=-0.2em]
        \item $N$ is the least positive integer such that $f$ is bi-$K(N)$-invariant,
        \item $\mathcal{F}_0(f)$ (resp. $\mathcal{F}_E(f)$) is the global family of automorphic cuspidal representations unramified at $\infty$ (resp. Hecke characters) cut by $f$ as in Section 1.4.5,
        \item For $\pi \in \mathcal{F}_0(f)$, the parameter $t_\pi$ is the spectral parameter at $\infty$, defined in \eqref{def:spectral parameter}, and $\mathcal{B}(\pi)$ is an orthonormal basis of pure tensors for $\pi^{K_\infty\times K(N)}$,
        \item For $\chi \in \mathcal{F}_E(f)$, the terms $\pi_{\chi,\chi^{-1}}$ and $E(\phi_{it})$ ($ t \in \mathbb{R},\phi\in \pi_{\chi,\chi^{-1}}$ smooth) are the unitarily induced representation / Eisenstein series described in Section 2.2.1, and $\mathcal{B}(\pi_{\chi,\chi^{-1}})$ is as in the cuspidal case,
        \item $H(m,n,c)\in \mathbb{C}$ and $\mathcal{C}(f)\subseteq\mathbb{Q}_+$ are the generalised Kloosterman sums and moduli space associated to $f$ given in Section \ref{section:kloosterman-weights},
        \item $u_\varphi$ (resp $u_{E(\phi_{it})}$) are the classical $\Gamma(N)$-Maass forms (resp. Eisenstein series)
        with Fourier coefficients $a_{u_\varphi}(m)$ (resp $a_{u_{E(\phi_{it})}}(m)$ ), as defined in \eqref{def:maasscoefficients},
        \item $h_\infty \in PW^{\operatorname{even}}(\mathbb{C})$, the even Paley-Wiener space as in \cite[\S 3.3]{KL-kuznetsov},  whereas $H_\infty^{-}$ is the transform of $h_\infty$ given by
        \begin{equation}\label{def:negative-H-transform}
            H_\infty^-(x)=\frac{1}{\pi}\int_{-\infty}^{\infty}K_{2it}(x)\sinh(\pi t)t h_\infty(t)dt 
        \end{equation}        
    \end{itemize}
\end{theorem}
Theorem \ref{thm:unrefinedBKformula} is often too general for applications, and we often desire better control on both the geometric and spectral sides. With this purpose in mind, we introduce two sets of conditions, called the geometric and spectral assumptions. \bigbreak

\noindent \textbf{Geometric assumption}:
 There exists $y \in \mathbb{Q}_+$ such that $\text{supp}(f) \subset a(y)^{-1}K_{\operatorname{fin }}a(y)$. \bigbreak 

For any $y \in \mathbb{Q}_+$ such that the Geometric assumption holds, we say that $y$ \emph{controls the support} of $f$. \bigbreak

For sake of notation, denote by $\mathcal{H}$ the space of functions on $G(\mathbb{A}_{\operatorname{fin}})$ that are locally constant and compactly supported.\\
\noindent \textbf{Spectral assumption}: The function $f \in \mathcal{H}$ is non-zero and admits a representation $f=\bigotimes_p f_p$ where $f_p$ is either a newform projector, or there exists $c \in \mathbb{Z}_{\geq 0}$ such that $f_p=\nu(p^c)1_{K_0(p^c)}$, where $\nu(n)=n\prod_{p|n}(1+p^{-1})$. \bigbreak

\noindent Given a cuspidal representation $\pi$ of $G(\mathbb{A})$ that factors as $\pi\simeq\bigotimes'_v\pi_v$, if $f$ satisfies the spectral assumptions, the operator $\pi(f)$ is an orthogonal projection onto the subspace
\begin{equation}\label{eq:V_f}
    V_{f} = \pi_\infty\otimes \bigotimes_{p: f_p \ \text{newform proj}}\mathbb{C}v_{0,p} \otimes \bigotimes_{p: f_p=v(p^c)1_{ZK_0(p^c)}}\pi_p^{K_0(p^c)}
\end{equation}
where $v_{0,p}$ is the unique newform of $\pi_p$, up to scalars. 

Following \cite[\S 7]{Petrow_2018}, if $\pi \in \mathcal{F}_0(f)$, there exists an orthonormal basis $\mathcal{B}_f(\pi)$ of the subspace $V_f^{K_\infty}$, and weights $w(\pi,f) \in \mathbb{C}$ such that for any $m_1,m_2 \in \mathbb{N}$ with $(m_1m_2,N)=1$ (here $N$ is the level of $f$)
\begin{equation}\label{def:discrete-weights}
    \sum_{\varphi \in \mathcal{B}_f(\pi)}a_{u_\varphi}(m_1)\overline{a_{u_{\varphi}}(m_2)}= w(\pi,f)\lambda_\pi(m_1)\overline{\lambda_\pi(m_2)}
\end{equation}
thus recovering the Hecke eigenvalues, normalised so that the Ramanujan-Petersson conjecture predicts $|\lambda_\pi(q)|\leq 2$ for $q$ prime coprime to $N$. The same results holds for $\pi \in \mathcal{F}_\kappa(f)$ and the right-isotypic space $V_f^{\kappa\times K(N)}$ consisting of functions $\varphi$ that are right-$K(N)$-invariant and such that $\varphi(gk_\theta)=\varphi(g)e^{i\kappa\theta}$. In an analogous fashion, for a Hecke character $\chi$, letting
\begin{equation}\label{def:cts-Hecke}
    \lambda_{\pi_{it}(\chi)}(m)=\sum_{ab=m}\chi(a)\chi^{-1}(b)(b/a)^{it}
\end{equation}
we have
\begin{equation}\label{def:continuous-weights}
    \sum_{\phi \in \mathcal{B}_f(\pi_{\chi,\chi^{-1}})}a_{u_{E(\phi_{it})}}(m_1)\overline{a_{u_{E(\phi_{it})}}(m_2)}=w(\pi_{it}(\chi),f)\lambda_{\pi_{it}(\chi)}(m_1)\overline{\lambda_{\pi_{it}(\chi)}(m_2)}
\end{equation}
where $\mathcal{B}_f(\pi_{\chi,\chi^{-1}})$ is a basis for $V_f^{K_\infty}$. With these further assumptions at hand, one may derive from Theorem \ref{thm:unrefinedBKformula} the following formula
\begin{theorem}[Refined BK formula]\label{thm:refined-opposite-signs}
    Let $f $ satisfy both the spectral and geometric assumptions. For $m_1,m_2 \in \mathbb{Z}$ with $m_1m_2<0$ and $(m_1m_2,N)=1$ we have that
    \begin{align}
     \sum_{\pi \in \mathcal{F}_0(f)}p(\pi)h_\infty(t_\pi)w(\pi,f)\lambda_\pi(|m_1|)\overline{\lambda_\pi(|m_2|)} \notag \\
     +\frac{1}{4\pi}\sum_{\chi \in \mathcal{F}_E(f)}p(\pi_{\chi,\chi^{-1}})\int_{-\infty}^{\infty}h_\infty(t)& w(\pi_{it}(\chi),f)\lambda_{\pi_{it}(\chi)}(|m_1|)\overline{\lambda_{\pi_{it}(\chi)}(|m_2|)}dt \\
    & = \sum_{c\in \mathcal{C}(f)}\frac{H(m_1,m_2,c)}{c}H_\infty^{-}\left(\frac{4\pi\sqrt{|m_1m_2|}}{c}\right) \notag
    \end{align}
    as absolutely convergent sums/integrals. Here
    \begin{itemize}[itemsep=-0.2em]
        \item $p(\pi)$ is the parity of $\pi$, as given in \eqref{def:parity},
        \item $h_\infty$ is now a test function that satisfies \eqref{def:testfunction-conditions} and $H_\infty^-$ is given in \eqref{def:negative-H-transform},
        \item $w(\pi,f)$ (resp $w(\pi_{it}(\chi),f)$) is the discrete (resp. continuous) weight in \eqref{def:discrete-weights} (resp. \eqref{def:continuous-weights}) and $\lambda_\pi(m)$ (resp. $\lambda_{\pi_{it}(\chi)}(m)$) is the Hecke eigenvalue of $\pi$ (resp. $\pi_{it}(\chi)$),
        \item All other notation is as in Theorem \ref{thm:unrefinedBKformula}.
    \end{itemize}
\end{theorem}

We also obtain a generalised version of the $\kappa=2$ Petersson formula. Unlike in the opposite sign BK formula, the $\kappa=2$ Petersson formula contains a diagonal term $\delta$ on the geometric side, which can be expressed as the product of local  weights $\delta_p$. Another difference is that the archimedean test function $f_\infty$ is now taken to be a smooth truncation of a matrix coefficient, see \eqref{def:truncated-test-function} for details.

\begin{theorem}[Petersson formula]\label{thm:p2-formula}
    Suppose that $f \in \mathcal{H}$ satisfies the geometric and spectral assumptions, for $m_1,m_2 \in \mathbb{N}$ with  and $(m_1m_2,N)=1$,
    \begin{align}
        \sum_{\pi \in \mathcal{F}_2(f)}w(\pi,f)\lambda_\pi(m_1) & \overline{\lambda_\pi(m_2)} \notag \\
        & =\frac{1}{4\pi}\delta_{m_1=m_2}\delta-\frac{1}{2}\sum_{c \in \mathcal{C}(\mathcal{F})}\frac{H(m_1,m_2,c)}{c}J_1\left(\frac{4\pi\sqrt{m_1m_2}}{c}\right)
    \end{align}
    where 
    \begin{itemize}[itemsep=-0.2em]
        \item $\delta=\prod_p\delta_p$ with the local weights $\delta_p$ given in \eqref{def:local-weights},
        \item $J_1$ is the Bessel function,
        \item $\mathcal{F}_2(f)$ is the family of automorphic cuspidal representations cut by $f$ that are weight 2 at $\infty$,
        \item All other notation is as in Theorems \ref{thm:unrefinedBKformula} and \ref{thm:refined-opposite-signs}.
    \end{itemize}
\end{theorem}

\begin{remark}
    It is possible to generalise Theorems \ref{thm:refined-opposite-signs} and \ref{thm:p2-formula} by dropping the assumption $(m_1m_2,N)=1$, but the resulting formula becomes more complicated. See \cite[Remark 4.3]{generalised}.
\end{remark}

\subsection{Parity equidistribution}
As a simple corollary of the BK formula for opposite signs is the equidistribution result for the parity of a given family cut by $f$. In this section we consider the archimedean test function $h_\infty$ to be one of
\begin{equation}\label{def:testfunction1}
    h_\infty(t)=\frac{t^2+\frac{1}{4}}{T^2}\left[\frac{1}{\cosh\left(\frac{t-T}{\Delta}\right)}+\frac{1}{\cosh\left(\frac{t+T}{\Delta}\right) }\right]
\end{equation}
where $1 \leq \Delta < T/100$, or
\begin{equation}\label{def:testfunction2}
    h_\infty(t) = \frac{t^2+\frac{1}{4}}{T^2}\exp\left(-\left(\frac{t}{T}\right)^2\right)
\end{equation}
Each of \eqref{def:testfunction1}  and \eqref{def:testfunction2} can be associated to a bi-$K_\infty$-invariant function $f_\infty$, supported on $G^+(\mathbb{R})$: the Plancherel formula gives us that
\begin{equation*}
    f_\infty(1)=\frac{1}{4\pi}\int_{-\infty}^{\infty}h_\infty(t)t\tanh(\pi t)dt
\end{equation*}
and by trivial estimates we have in the respective cases
\begin{equation}
    f_\infty(1)\asymp\Delta T \hspace{1cm} \text{or} \hspace{1cm} f_\infty(1)\asymp T^2
\end{equation}

\begin{corollary}[Parity equidistribution]\label{cor:equidistribution}
    Let $h_\infty$ be one of \eqref{def:testfunction1} or \eqref{def:testfunction2}, if $f \in \mathcal{H}$ is as in Theorem \ref{thm:refined-opposite-signs}, then for $m \in \mathbb{N}$ with $(m,N)=1$, and $\epsilon \in \{-1,1\}$,
    \begin{equation}
        \sum_{\substack{\pi \in \mathcal{F}_0(f) \\ p(\pi)=\epsilon}}h_\infty(t_\pi)w(\pi,f)|\lambda_\pi(m)|^2=\frac{1}{2}f_\infty(1)\delta+O\left (\frac{f_\mathbb{A}(1)m^2}{T^2\kappa^2(\mathcal{F})} \exp\left(\frac{16\pi^2 m}{k(\mathcal{F})}\right)\right)
    \end{equation}
    where $\kappa(\mathcal{F})$ is the geometric conductor of $f$ (see \ref{section:kloosterman-weights}) and $f_\mathbb{A}(1)=f_\infty(1)f(1)$ with all other notation as before.
\end{corollary}
It is likely that the error term may be significantly improved, and that the weights may be altogether dropped, but the corollary above is meant to be a simple demonstration of the usefulness of such formulas, so any improvement is beyond the aim of this paper. We refer the reader to Section \ref{section:corollary} for a proof of the result above.

\subsection{Notation and conventions}
\emph{1.4.1 Fields}: Throughout the paper we work over the rationals $\mathbb{Q}$. We write $\mathbb{Q}_p$ for the field of $p$-adic rationals with ring of integers $\mathbb{Z}_p$, equipped with the absolute value $|\cdot|_p$. Let $\mathbb{A}$ and $\mathbb{A}_\text{fin}$ be the adeles and the finite adeles over $\mathbb{Q}$ respectively, and $\widehat{\mathbb{Z}}=\prod_p\mathbb{Z}_p$ the maximal compact open subgroup of $\mathbb{A}_\text{fin}$. \bigbreak

\noindent \emph{1.4.2 Groups}: Write $G = \text{PGL}_2$ and $[G]=G(\mathbb{Q})\backslash G(\mathbb{A})$. Denote further $K_\infty = \text{PSO}_2(\mathbb{R})$ with matrices $k_\theta=\begin{pmatrix}
    \cos \theta & \sin \theta \\ -\sin \theta & \cos \theta
\end{pmatrix} $ and $K_p = G(\mathbb{Z}_p)$, the standard maximal compact subgroup of $G(\mathbb{Q}_p)$. Define $G^+(\mathbb{R})=\{g\in G(\mathbb{R}):\det g >0\}$ and $K_\text{fin}=\prod_pK_p$. We also identify the open compact subgroups of $K_\text{fin}$:
$$K_0(N) = \left\{\begin{psmallmatrix}
    a & b \\ c & d
\end{psmallmatrix} \in K_\text{fin}\ \vert \ c\ \equiv 0\ \text{mod}\ N \right\}$$
$$ K(N) = \left\{\begin{psmallmatrix}
    a & b \\ c & d
\end{psmallmatrix} \in K_\text{fin}\ \vert \ a\equiv d\ \equiv 1\ \text{mod}\ N, \ b\equiv c \equiv 0 \ \text{mod} \ N  \right\}$$
Let $N \subset B \subset G$ denote the standard upper unipotent and Borel subgroups of $G$, and $A$ the subgroup of matrices of the form $a(y)=\begin{psmallmatrix}
    y & 0 \\ 0 & 1
\end{psmallmatrix}$ with $y$ in a commutative ring $R$, and let $n(x)=\begin{psmallmatrix}
    1 & x \\ 0 & 1
\end{psmallmatrix}$ for $x \in R$. Finally we take $[N]=N(\mathbb{Q})\backslash N(\mathbb{A})$. \bigbreak 

\noindent \emph{1.4.3 Haar Measures:} We take Haar measures as in \cite[\S 7]{Traces}. In particular choose $dx$ to be the Lebesgue measure of $\mathbb{R}$ and $d^\times x=dx/|x|$ on $\mathbb{R}^\times$, whereas at the finite places we choose $dx$ (resp $d^\times x$) to be so that $\mathbb{Z}_p$ has volume $1$ (resp. $\mathbb{Z}_p^\times$ has volume $1$). We also let $dk$ denote the probability measure on $K_\infty$, and take the measures on $N(\mathbb{R})$ and $A(\mathbb{R})$ to be those induced by $dx$ and $d^\times x$; these determine the Haar measure on $G(\mathbb{R})$ via the Iwasawa decomposition. Choose a Haar measure $dg$ on $G(\mathbb{Q}_p)$ so that $K_p$ has volume 1. Equip $\mathbb{A}$ and $G(\mathbb{A})$ with the restricted product measure and give the quotient measure to $[G]$ and $[N]$. \bigbreak 

\noindent \emph{1.4.4 Characters:} The trace formula requires weights in the form of characters of $N(\mathbb{Q})\backslash N(\mathbb{A})$. Since the group is isomorphic to $\mathbb{Q}\backslash\mathbb{A}$, we consider characters of the latter instead. Let $\theta(x)=\prod_v\theta_v(x_p)$ be the standard additive character $\mathbb{Q}\backslash \mathbb{A} \to \mathbb{C}$
defined by
\begin{equation}
    \theta_v(x)=\begin{cases}
    e(-x) \qquad \text{if } v=\infty \\
    e(r_p(x)) \hspace{0.43cm} \text{if } v <\infty
    \end{cases}
\end{equation}
where $r_p(x) \in \mathbb{Q}$ denotes the $p$-principal part of $x\in \mathbb{Q}_p$, a rational number with $p$-power denominator such that $x\in r_p(x)+\mathbb{Z}_p$. For $m \in \mathbb{Q}$ we denote
\begin{equation}
    \theta_m(x)=\theta(-mx) = \overline{\theta(mx)}
\end{equation}
All characters of $\mathbb{Q}\backslash \mathbb{A}$ arise this way. \bigbreak 

\noindent \emph{1.4.5 Hecke algebras:} Let $\mathcal{H}_p$ the \emph{local Hecke algebra}, that is the space of locally constant, compactly supported functions on $G(\mathbb{Q}_p)$, and let $G(\mathbb{Q}_p)^\wedge$ denote the generic unitary dual of $G(\mathbb{Q}_p)$, the space of classes of smooth generic irreducible unitary representations on a $\mathbb{C}$-vector space. We say $f \in \mathcal{H}_p$ is a \emph{newform projector} if for every $\pi\in G(\mathbb{Q}_p)^\wedge$ the operator $\pi(f)$ is either $0$ or projects onto the newform line of $\pi$. 

We construct the \emph{global Hecke algebra} by $\mathcal{H}=\bigotimes_{p<\infty} \mathcal{H}_p$ and given $f \in \mathcal{H}$ we define the \emph{level} of $f$ to be the least $N \in \mathbb{N}$ such that $f$ is bi-$K(N)$-invariant. 
We also define the global families $\mathcal{F}_0(f)$  (resp. $\mathcal{F}_\kappa(f)$) for the set of cuspidal automorphic representation $\pi$ that are spherical (resp. isomorphic to the discrete series $\mathcal{D}(\kappa)$) at $\infty$ and such that $\pi(f): V_\pi \to V_\pi$ is not the zero map. We define similarly the continuous family
\begin{equation}\label{def:continuous-family}
    \mathcal{F}_E(f)=\{\chi \in (\mathbb{Q}^\times \backslash \mathbb{A}^1)^\wedge: \ \text{there exists } t\in\mathbb{R} \ \text{with} \ \pi_{\chi|\cdot|^{it},\chi^{-1}|\cdot|^{-it}}(f)\neq 0 \}
\end{equation}
For more details and additional notation see the Sections 2.1.1 and 2.2.1.

\subsection{Acknowledgments}
The author thanks I. Petrow for a careful reading of the first draft of this paper. \\ This work was supported by the Engineering and Physical Sciences Research Council [grant
numbers EP/T517793/1, EP/W524335/1].

\section{Preliminary results}
We introduce further notation and notions needed for the proof of our main results. These mostly concern the necessary representation theoretic background, Maass/holomorphic forms on the groups $\Gamma(N)$ and $\Gamma_0(N)$, as well as properties of the generalised Kloosterman sums. 

\subsection{Local representation theory}
\emph{2.1.1 Archimedean }$(\mathfrak{g},K_\infty)$\emph{-modules:} Let $\mathfrak{g}$ be the complexified Lie algebra of $G(\mathbb{R})$. Since this paper is mostly concerned with analysis at the archimedean place, we begin with a description of the irreducible infinite dimensional $(\mathfrak{g},K_\infty)$-modules of $G(\mathbb{R})$. For $\epsilon \in \{0,1\}$ and $t \in \mathbb{R} \cup(-1/2,1/2)i $, we let $\chi=\text{sgn}^\epsilon |\cdot|^{it}$ and denote by $\pi:=\pi_{\chi,\chi^{-1}}$ the space of smooth right-$K_\infty$ finite functions on $G(\mathbb{R})$ that satisfy
\begin{equation}\label{eq:principal-series-equation}
    \phi\left(\begin{pmatrix}
    a & b \\ 0 & d
\end{pmatrix}g\right)=\chi(a)\chi^{-1}(d)\Bigr|\frac{a}{d}\Bigr|^{1/2}\phi(g)
\end{equation}
equipped with the inner product
\begin{equation}\label{eq:principal-series-inner-prod}
    \langle\phi,\psi \rangle=\int_{K_\infty}\phi(k)\overline{\psi(k)}dk
\end{equation}
A basis for $\pi$ consists of vectors $\{\phi_k: k \in 2\mathbb{Z}\}$ where $\phi_k$ is defined via $\phi_k(k_\theta)=e^{ik\theta}$. We call $\epsilon$ the \emph{parity} of $\pi$ and the parameter $t$ the \emph{spectral parameter} of $\pi$. We also say a vector has \emph{weight} $k$ if it transforms by the scalar $e^{ik\theta}$ under the action of $\pi(k_\theta)$. \bigbreak 

\noindent Now fix an integer $\kappa \geq 2$, and let $\pi:=\mathcal{D}(\kappa)$ be the \emph{weight $\kappa$ discrete series} equipped with an inner product as described in \cite[\S11.7]{Traces}. Let $v_0$ be the lowest weight vector of $\mathcal{D}(\kappa)$ normalised with respect to the inner product, we consider the matrix coefficient $\Phi_{\kappa,v_0}(g):=\langle \pi(g)v_0,v_0 \rangle$ and define
$$f_\infty(g)=\frac{1}{||\Phi_{\kappa,v_0}||^2_{L^2(G(\mathbb{R}))}}\overline{\Phi_{\kappa, v_0}(g)}$$
Explicitly for $g=\begin{psmallmatrix}
    a & b \\ c & d
\end{psmallmatrix}$ we have
\begin{equation}\label{def:matrix-coeff}
    f_\infty(g) = 
    \begin{cases}
        \displaystyle\frac{\kappa-1}{4\pi} \frac{\det(g)^{\kappa/2}(2i)^\kappa}{(-b+c+(a+d)i)^\kappa} \quad \text{if } \det(g)> 0 \\
        \hspace{2cm} 0 \hspace{2.2cm} \text{otherwise}
    \end{cases}
\end{equation}
The function $f_\infty$ lies in $L^2(G(\mathbb{R}))$ for $\kappa\geq 2$, as well as in $L^1(G(\mathbb{R}))$ when $\kappa>2$. In the latter case the operator $\pi(f_\infty)$ is a projector onto the newform line of $\pi$. For details, see e.g. \cite[\S 14]{Traces}. \bigbreak

\noindent \emph{2.1.2 Non-archimedean principal series:} Let $p$ be prime, for $\chi$ a character of $\mathbb{Q}_p^\times$, not necessarily of finite order, we define the principal series representation $\pi:=\pi_{\chi,\chi^{-1}}$ consisting of locally constant functions that satisfy \eqref{eq:principal-series-equation} with inner product obtained by replacing $K_\infty$ with $K_p$ in \eqref{eq:principal-series-inner-prod}. \bigbreak

\subsection{Global representation theory}
The setting of the trace formula is the Hilbert space $L^2([G])$, on which $G(\mathbb{A})$ acts on the right. Denote by $L_0^2([G])$ the subspace of cuspidal functions. The space $L^2([G])$ admits a decomposition into irreducible unitary representations $\pi$ of $G(\mathbb{A})$, which in turn  decompose as a restricted tensor product $\pi = \bigotimes'_v \pi_v$. Following \cite[\S 2.2.1]{Michel_2010}, we say a global representation is \emph{standard} if it appears in the spectral decomposition of $L^2([G])$. We emphasize here that at the archimedean place by a representation we are really talking about the underlying $(\mathfrak{g},K_\infty)$-module, but for convenience we abuse terminology throughout. \bigbreak

\noindent \emph{2.2.1 Induced model and Eisenstein series:} For a global finite order Hecke character $\chi$  and $s \in \mathbb{C}$, we construct the global representation $\pi_{\chi|\cdot|^s, \chi^{-1}|\cdot|^{-s}}=\bigotimes'_v\pi_{\chi_v|\cdot|_v^{s},\chi_v^{-1}|\cdot|_v^{-s}}$.  Sometimes we also use the notation $\pi_{s}(\chi)$ for $\pi_{\chi|\cdot|^{s},\chi^{-1}|\cdot|^{-s}}$.\\
Given a smooth vector $\phi \in \pi_{\chi,\chi^{-1}}$ we may construct an Eisenstein series
\begin{equation}\label{eq:Eisenstein-series}
    E(\phi_s,g)=\sum_{\gamma \in B(\mathbb{Q})\backslash G(\mathbb{Q})}\phi_s(\gamma g) \qquad (g \in G(\mathbb{A}))
\end{equation}
where the map $\pi_{\chi,\chi^{-1}}\to \pi_{\chi|\cdot|^s,\chi^{-1}|\cdot|^{-s}} $ with $\phi\mapsto \phi_s$ is an isomorphism of vector spaces given by $\phi_s(g)=|a/d|^s\phi(g)$ for $g=\begin{psmallmatrix}
    1 & x \\ & 1
\end{psmallmatrix}\begin{psmallmatrix}
    a & \\ & d
\end{psmallmatrix}k$ for $a,d \in \mathbb{A}^\times, x \in \mathbb{A}, k \in K_\infty\times K_\text{fin}$.
The sum on the right of \eqref{eq:Eisenstein-series} is absolutely convergent for $Re(s)>1/2$, but $E(\phi_s,g)$ admits a meromorphic continuation to $\mathbb{C}$, holomorphic on the line $Re(s)=0$ (see e.g. \cite[\S 3.7]{Bump}). \bigbreak

\noindent \emph{2.2.2 Orthocomplement of the cuspidal space:} We may write $L^2([G])=L_0^2([G])\oplus L^2_\text{res}([G])\oplus L^2_\text{cts}([G])$ where $L^2_\text{res}([G])$ and $L^2_\text{cts}([G])$ are the \emph{residue} and \emph{continuous} spectrum respectively. Given $\chi$ quadratic, denoting $\phi_\chi(g)=\chi(\det g)$ we have the decomposition
$$L^2_\text{res}([G])=\bigoplus_{\chi^2=1}\mathbb{C}\phi_\chi$$
whereas the continuous spectrum admits a decomposition in terms of (global) principal series
$$L^2_\text{cts}([G])=\bigoplus_{\substack{\chi \ \text{finite}\\ \text{order}}}\int_{\mathbb{R}}\pi_{it}(\chi)$$
For further details, see for instance \cite[\S 4D]{Gelbart-Jacquet}. \bigbreak

\noindent \emph{2.2.3 Miscellaneous:} For a standard generic automorphic representation $\pi$ of $G(\mathbb{A})$ that is spherical at $\infty$, we denote
\begin{equation}\label{def:parity}
    p(\pi)=(-1)^\epsilon \quad \text{where} \quad \pi_\infty\simeq \pi(\text{sgn}^\epsilon|\cdot|^{it},\text{sgn}^\epsilon |\cdot|^{-it})
\end{equation}
and 
\begin{equation}\label{def:spectral parameter}
    t_\pi=t \quad \text{where} \quad \pi_\infty\simeq \pi(\text{sgn}^\epsilon|\cdot|^{it},\text{sgn}^\epsilon |\cdot|^{-it})
\end{equation}
which we refer to as the parity and spectral parameter of $\pi$. 

\subsection{Whittaker/Fourier coefficients} Let $\pi$ be a standard generic automorphic representation of $G(\mathbb{A})$, and let $\varphi \in \pi$ be a smooth vector. We introduce the constant term and the Whittaker period for $\varphi$, defined respectively as
$$\varphi_N(g)=\int_{\mathbb{Q}\backslash \mathbb{A}}\varphi(n(x)g)dx \qquad \text{and}\qquad W_\varphi(g)=\int_{\mathbb{Q}\backslash \mathbb{A}}\varphi(n(x)g)\overline{\theta (x)}dx$$
Then for almost all $g \in G(\mathbb{A})$, the vector $\varphi$ admits a Fourier expansion of the form
\begin{equation}
    \varphi(g)=\varphi_N(g)+\sum_{y \in \mathbb{Q}^\times}W_\varphi(a(y)g)
\end{equation}
Recall that $\pi^{\kappa\times K(N)}$ is the isotypic subspace of $\pi$ defined between \eqref{def:discrete-weights} and \eqref{def:cts-Hecke}. Given $\varphi \in \pi^{K_\infty \times K(N)}$ (resp. $\varphi \in \pi^{\kappa \times K(N)}$) we can construct a Maass form/ Eisenstein series (resp. weight $\kappa$ holomorphic form) on $\Gamma(N)$ via
$$u_\varphi(x+iy)=\varphi(\begin{psmallmatrix}
    y & x \\ & 1
\end{psmallmatrix} \times 1_\text{fin}) \qquad \left(\text{resp. } u_\varphi(x+iy)=y^{-\kappa/2}\varphi(\begin{psmallmatrix}
    y & x \\ & 1
\end{psmallmatrix} \times 1_\text{fin})\right)$$
for $x+iy \in \mathbb{H}$. Then $u$ admits the Fourier expansion
\begin{equation}
    u(x+iy)=\sum_{n \in \mathbb{Z}}a_u(n/N,y)e(\frac{n}{N}x) \qquad a_u(n/N,y)=\frac{1}{N}\int_0^N u(x+iy)e\left(-\frac{n}{N}x\right)dx
\end{equation}
(for holomorphic forms the summation can be restricted over $n \in \mathbb{N}$). Let $u$ be a Maass form/ Eisenstein series, then
for $m = n/N \neq 0$, we may define the rescaled Fourier coefficients $a_u(m)$ by
\begin{equation}\label{def:maasscoefficients}
    \frac{a_u(m)}{\sqrt{|m|}}W(my)=a_u(m,y) \qquad \text{where } W(y)=2(\text{sgn}y)^\epsilon \left(\frac{\cosh (\pi t)}{\pi}\right)^{1/2}\sqrt{|y|}K_{it}(2\pi |y|)
\end{equation}
whereas if $u$ is a holomorphic form
\begin{equation}\label{def:holomorphic-coefficients}
    \frac{a_u(m)}{\sqrt{m}}W(my)=y^{\kappa/2}a_u(m,y) \qquad \text{where } W(y)=\begin{cases} \left(\frac{(4\pi y)^\kappa}{\Gamma(\kappa)}\right)^{1/2}e^{-2\pi y} \quad \text{if }y>0 \\
        0 \hspace{2.9cm} \text{if }y<0
    \end{cases}
\end{equation}

\begin{proposition} Let $\pi$ be a standard generic representation of $G(\mathbb{A})$, and let $\varphi \in \pi^{K_\infty\times K(N)}$ if $\pi$ is spherical at $\infty$, or $\varphi \in \pi^{\kappa\times K(N)}$ if $\pi_\infty\simeq \mathcal{D}(\kappa)$.\\
Let $m \in\mathbb{Q}^\times$ and $y \in \mathbb{R}^\times$, the Fourier coefficients and the Whittaker coefficients are related by
    \begin{equation}\label{eq:whittaker-fourier}
        W_\varphi(a(-my))=\begin{cases}
            a_{u_\varphi}(m,y) \quad \text{if } m \in \frac{1}{N}\mathbb{Z} \\
            0 \hspace{1.7cm} \text{otherwise}
        \end{cases}
    \end{equation}
\end{proposition}
\begin{proof}
    See for instance \cite[Lemma 3.6]{Gelbart}.
\end{proof}
The results above still hold if we consider $K_0(N)$ instead of $K(N)$, with the only differences being that $u_\varphi$ is now a Maass/Eisenstein/holomorphic form on $\Gamma_0(N)$ and we can replace $\tfrac{1}{N}\mathbb{Z}$ with $\mathbb{Z}$ in the Fourier expansion. \bigbreak

\subsection{Distribution results for spectral parameters}
As previously discusses, the newform matrix coefficient of $\mathcal{D}(\kappa)$ fails to be square integrable for $\kappa=2$. We circumvent this issue in Section \ref{section:holomorphic} by taking a smooth compactly supported cutoff and apply a limiting argument. The tradeoff is in the form of extra contribution from the cuspidal and continuous space; such contribution needs to be estimated in order to obtain cancellation in the limit.

Throughout this section let $\{u_j\}$ be an orthonormal basis of cusp Maass forms on $\Gamma_0(N)$, with spectral parameter $t_j$.

\begin{proposition}[Weak Weyl's law] \label{prop:weyl-law}Let $\{u_j\}$ and $t_j$ be defined as above then for any $T>0$
\begin{equation}
    |\{j:|t_j|\leq T\}|\ll T^2
\end{equation}
with the implicit constant depending on $N$.
\begin{proof}
    See \cite[Theorem 1.1]{Donnelly1982}.
\end{proof}
\end{proposition}
Using the trivial bound in \cite[Theorem 3.2]{Spectralmethods}, we have $|a_{u_j}(m)|^2\ll_{N,m}(1+|t_\pi|)$ therefore we may deduce the following result on sums of Fourier coefficients.
\begin{lemma}\label{lemma:maass-growth}
    Let $\{u_j\}$ and $t_j$ be defined as above, let $T\geq 1$, $m \in \mathbb{N}$, then
    \begin{equation}
        \sum_{|t_j|\leq T}|a_{u_j}(m)|^2\ll T^3
    \end{equation}
    where the implied constant only depends on $m,N$.
\end{lemma}
We highlight that these estimates are very weak and in no way optimal, as we wish to prove the theorems of this paper with the mildest possible assumptions with the hope that our techniques might apply to other contexts.

\subsection{Kloosterman sums and local weights}\label{section:kloosterman-weights}
Given $f \in \mathcal{H}$ non-zero, for $m,n \in \mathbb{Q}$ and $c\in \mathbb{Q}_+$ we can attach a \emph{generalised Kloosterman sum} defined as
\begin{equation}\label{def:kloosterman-sum}
    H(m,n,c)=\iint_{\mathbb{A}_\text{fin}^2}f(\begin{pmatrix}
        1 & -t_1 \\ & 1
    \end{pmatrix}\begin{pmatrix}
        & -c^{-2} \\ 1 &
    \end{pmatrix}\begin{pmatrix}
        1 & t_2 \\ & 1
    \end{pmatrix})\theta_\text{fin}(mt_1-n t_2)dt_1 dt_2
\end{equation}
We say that $c\in \mathbb{Q}_+$ is an \emph{admissible modulus} if there exists a pair $(m,n) \in \mathbb{Q}^2$ such that $H(m,n,c)\neq 0$, and write $\mathcal{C}(\mathcal{F})$ for the set of admissible moduli. By \cite[Proposition/Definition 1.3]{generalised}, if $f=\bigotimes_p f_p$ is a pure tensor so that the support of $f_p$ is in $\{g \in G(\mathbb{Q}_p): v_p(\det g)\in 2\mathbb{Z}\}$, there exists a $q'\in \mathbb{Q_+}$ such that $\mathcal{C}(f) \subseteq q'\mathbb{Z}$. If the set $\mathcal{C}(\mathcal{F})$ is non-empty, we write $\kappa(\mathcal{F})$ for the maximal such $q'$, and call it the \emph{geometric conductor} of $f$.   \bigbreak

\begin{proposition}\label{prop:kloosterman-properties} Suppose that $f \in \mathcal{H}$ satisfies the Geometric and Spectral assumptions, then for $(m,n)\in \mathbb{Q}^2 $ and $c\in \mathbb{Q}_+$:
\begin{enumerate}[itemsep=0em, label=(\arabic*)]
    \item The set of admissible moduli $\mathcal{C}(f)$ is non-empty,
    \item If $f$ is a function of level $N$, it is bi-$K_0(N)$-invariant,
    \item The sums $H(m,n,c)$ satisfy the trivial bound
    \begin{equation}\label{eq:trivial-bound}
        |H(m,n,c)|\leq c\kappa(\mathcal{F)}f(1),
    \end{equation}
    \item $H(m,n,c)=0$ unless $(m,n)\in \mathbb{Z}^2$ and $c \in \mathbb{N}$.
\end{enumerate}
\end{proposition}
\begin{proof}
    All proofs are contained in \cite{generalised}. More precisely, see Lemmas 3.5 and 3.6 for \emph{(1)}, Section 1.3.2 and Proposition 4.1 for \emph{(2)}, Theorem 3.8 and Lemma 4.7 for \emph{(3)}, and Lemma 4.6 for \emph{(4)}.
\end{proof}

Let $\Sigma$ the Borel $\sigma$-algebra of $G(\mathbb{Q}_p)^{\wedge}$ with respect to the Fell topology, and $\mu$ a Haar measure on $G(\mathbb{Q}_p)$, there exists (see e.g. \cite[Chapter 18]{dixmier1969}) a unique $\sigma$-finite measure $\widehat{\mu}$ on $(\overline{G}(\mathbb{Q}_p)^{\wedge},\Sigma)$ such that
\begin{equation}
    \int_{{G}(\mathbb{Q}_p)}|f(g)|^2d\mu = \int_{{G}(\mathbb{Q}_p)^\wedge}||\pi(f)||_{HS}^2 \  d\widehat{\mu}
\end{equation}
for all (local) $f$ compactly supported on $G(\mathbb{Q}_p)$. We refer to $d\widehat{\mu}$ as the \emph{Plancherel measure}. \bigbreak 

We may now introduce the local weights $\delta_p$ appearing in Theorem \ref{thm:p2-formula}. For $\pi \in G(\mathbb{Q}_p)^\wedge$ let
\begin{equation}
    \mathcal{L}_\pi(1) =\begin{cases}
        \frac{(1-p^{-2})}{(1-e^{2i\theta}p^{-1})(1-p^{-1})(1-e^{-2i\theta}p^{-1})} \quad \text{if } \pi \simeq \pi(|\cdot|_p^{i\theta/\log p},|\cdot|_p^{-i\theta/\log p}) \\
        (1+1/p)^{-1} \hspace{3.1 cm} \text{if } c(\pi)=1\\
        ( 1-1/p) \hspace{3.5cm} \text{if } c(\pi)\geq 2
    \end{cases}
\end{equation}
where in the first line either $\theta\in[0,\pi]$, or $\theta=i\tau \log p$ or $\pi=i\tau\log p$ with $\tau \in (0,1/2)$.
If $f_p$ is a newform projector (resp. $f_p=\nu(p^c)1_{K_0(p^c)}$) then
\begin{equation}\label{def:local-weights}
    \delta_p= \int_{\mathcal{F}_p(f)}\frac{1}{\mathcal{L}_\pi(1)}d\widehat{\mu}(\pi) \qquad \left(\text{resp. } \delta_p = \int_{\mathcal{F}_p(f)}\text{dim}\pi ^{K_0(p^c)}d\widehat{\mu}(\pi) \right)
\end{equation}
where $\mathcal{F}_p(f)=\{\pi\in G(\mathbb{Q}_p)^\wedge:\pi(f):V_\pi \to V_\pi \text{ is not zero}\}$ is the local family cut by $f$.

\section{An overview of the relative trace approach}
The starting point of the adelic relative trace approach is the pre-trace formula, which allows us to obtain a spectral expansion of the kernel operator attached to a well behaved function on $G(\mathbb{A})$. We begin this chapter by defining the main terms on the geometric and spectral side of the pre-trace formula. We set up the main ideas to prove our cases of interest by providing an outline of the proof for the Bruggeman-Kuznetsov formula in the same sign case. Working formally for the moment, we will discuss the refinement technique. We also address the derivation of the Petersson formula for weight $\kappa>2$. 

Both discussions follow the work of Knightly and Li. Brief remarks on the changes we need to introduce are provided throughout, but we develop the additional theory with more rigour in Chapter \ref{section:opposite-signs} and Chapter \ref{section:holomorphic}.

\subsection{The pre-trace formula}
Let $f_\infty \in {C}_c^\infty(G(\mathbb{R}))$, the space of smooth, compactly supported functions on $G(\mathbb{R})$. For $f \in \mathcal{H}$, which we assume satisfies the same conditions of Theorem
\ref{thm:refined-opposite-signs}, the global function
$f_\mathbb{A}=f_\infty \otimes f$ lies in the Schwartz space $C_c^\infty(G(\mathbb{R}))\otimes \mathcal{H}$.
Any such global function admits a geometric kernel
\begin{equation}\label{def:geometrickernel}
    K_\text{geom}(x,y)=\sum_{\gamma \in G(\mathbb{Q})}f_\mathbb{A}(x^{-1}\gamma y) \quad ((x,y) \in G(\mathbb{A})\times G(\mathbb{A}) )
\end{equation}
where the sum is absolutely and uniformly convergent on compacta of $[G]^2$ (it is in fact a locally finite sum).

The pre-trace formula provides an alternative form of $K_\text{geom}$ via a spectral expansion. Given a cuspidal representation $\pi$ of $G(\mathbb{A})$, let $\mathcal{B}(\pi)$ any orthonormal basis consisting of $K_\infty$ isotypic vectors that respect the decomposition $\pi^{K(N)}\times (\pi^{K(N)})^\bot$. Similarly, for $\chi$ a Hecke character, denote by $\mathcal{B}(\chi,\chi^{-1})$ any orthonormal basis for $\pi_{\chi,\chi^{-1}}$ that respects the same two conditions.

\begin{theorem}[Pre-trace formula]\label{thm:pretrace}
    For any $f_\mathbb{A} \in C_c^\infty(G(\mathbb{R}))\otimes \mathcal{H}$ and $(x,y) \in G(\mathbb{A})\times G(\mathbb{A})$, we have
    \begin{equation}
        K_{\operatorname{geom}}(x,y)=K_{\operatorname{cusp}}(x,y)+K_{\operatorname{res}}(x,y)+K_{\operatorname{cts}}(x,y)
    \end{equation}
    where $K_{\operatorname{geom}}$ is defined as in (\ref{def:geometrickernel}), 
    $$K_{\operatorname{cusp}}(x,y)=\sum_{\pi \ \text{cuspidal}}\sum_{\varphi \in \mathcal{B}(\pi)}R(f_\mathbb{A})\varphi(x)\overline{\varphi(y)}$$
    where  each $\pi$ in the summand has trivial central character,
    $$K_{\operatorname{res}}(x,y)=\frac{3}{\pi}\sum_{\chi^2=1}\chi(\det x)\overline{\chi(\det y)}\int_{G(\mathbb{A})}f_\mathbb{A}(g)\chi(\det g)dg$$
    where sum ranges over quadratic Hecke characters, and
    $$K_{\operatorname{cts}}(x,y)=\frac{1}{4\pi}\sum_{\chi}\sum_{\phi \in \mathcal{B}(\chi,\chi^{-1})}\int_{-\infty}^{\infty}E(\pi_{it}(f_\mathbb{A})\phi_{it},x)\overline{E(\phi_{it},y)}dt$$
    where the outer sums ranges over the finite order Hecke characters. The right hand side converges absolutely and uniformly on compacta of $[G]^2$. 
\end{theorem}
\begin{proof}
    See \cite[Theorem 2.2]{LuoPiWu}. See also \cite[Corollary 6.12]{KL-kuznetsov}.
\end{proof}

\subsection{The BK formula in the same signs case}\label{section:same-sign}
In this section we sketch the derivation of the BK formula in the case that $m_1,m_2$ are both positive. The case in which $m_1,m_2$ are both negative follows similarly, with only very minor modifications, therefore we choose to omit it. We assume throughout that $f \in \mathcal{H}$ satisfies the conditions in Theorem \ref{thm:unrefinedBKformula}. For the archimedean test function we make use of the following proposition.

\begin{proposition}\label{prop:paley-wiener-isom}
    Let $C^\infty_c(G^+(\mathbb{R})//K_\infty)$ be the space of smooth, compactly supported functions with support in $G^+(\mathbb{R})$, which are also bi-$K_\infty$-invariant. There is a linear isomorphism
    $$ C_c^\infty(G^+(\mathbb{R})//K_\infty) \longrightarrow PW^{\operatorname{even}}(\mathbb{C}) \qquad f_\infty \longmapsto h_\infty$$
    where $PW^{\operatorname{even}}(\mathbb{C})$ is the even Paley-Wiener space as in \cite[\S 3.3]{KL-kuznetsov}.
\end{proposition}
\begin{proof}
    See \cite[Proposition 3.6]{KL-kuznetsov}.
\end{proof}

Let $K_{\operatorname{geom}}$ be the kernel associated to $f_\infty\otimes f$ with $f_\infty\in C_c^\infty(G^+(\mathbb{R})//K_\infty) $. For $m_1,m_2 \in \mathbb{Q}_+$ and $y_1,y_2 >0$, we aim to evaluate 
\begin{equation}
    I = \frac{1}{\sqrt{y_1y_2}}\iint_{[N]^2}K_\text{geom}(\begin{psmallmatrix}
        y_1 & 0 \\ 0 & 1
    \end{psmallmatrix}n_1,\begin{psmallmatrix}
        y_2 & 0 \\ 0 & 1
    \end{psmallmatrix}n_2)\overline{\theta_{m_1}(n_1)}\theta_{m_2}(n_2)dn_1dn_2
\end{equation}
in two ways, following Theorem \ref{thm:pretrace}. Since $[N]^2$ is compact, we may apply Fubini's theorem to exchange integration with the sums that define $K_\text{cusp}, K_\text{res}$ and $K_\text{cts}$, resulting in a decomposition of the form
\begin{equation}\label{eq:decomposition}
    I = I_\text{cusp}+I_\text{res}+I_\text{cts}
\end{equation}

\noindent \emph{3.2.1 Cuspidal contribution:} Swapping the order of summation and integration 
$$I_\text{cusp}=\frac{1}{\sqrt{y_1y_2}}\sum_{\pi \in \mathcal{F}_0(f)}h_\infty(t_\pi)\sum_{\varphi \in \mathcal{B}(\pi)}W_{\pi(f)\varphi}(a(-m_1y_1))\overline{W_\varphi(a(-m_2y_2))}$$
where we recall all $\pi$ in the summation are cuspidal and with trivial central character. Since $f$ is bi-$K(N)$-invariant, the vector $\pi(f)\varphi \in \pi^{K_\infty \times K(N)}$ and thus admits a description as a Maass form $u_{\pi(f)\varphi}$.
Therefore if $m_1,m_2 \in \tfrac{1}{N}\mathbb{Z}$ then
\begin{equation}\label{eq:cusp-unrefined-samesigns}
        I_\text{cusp}=\frac{4}{\pi}\sum_{\pi \in \mathcal{F}_0(f)}h_\infty(t_\pi)(\cosh (\pi t_\pi))K_{it_\pi}(2\pi m_1y_1)K_{it_\pi}(2\pi m_2 y_2) \sum_{\varphi \in \mathcal{B}(\pi)}a_{u_{\pi(f)\varphi}}(m_1)\overline{a_{u_\varphi}(m_2)}  
\end{equation}
Upon setting the constraint $w=m_1y_1 = m_2y_2$ and integrating with respect to $w$, equation \eqref{eq:cusp-unrefined-samesigns} becomes
\begin{equation}\label{eq:cusp-refined-samesigns}
    \int_0^\infty I_\text{cusp}(w)dw = \frac{1}{2}\sum_{\pi \in \mathcal{F}_0(f)}h_\infty(t_\pi)\sum_{\varphi \in \mathcal{B}(\pi)}a_{u_{\pi(f)\varphi}}(m_1) \overline{a_{u_\varphi}(m_2)}
\end{equation}
Note that if $m_1,m_2$ had opposite signs, equation \eqref{eq:cusp-refined-samesigns} would be the same except for a correction via the parity $p(\pi)$. \bigbreak 

\noindent \emph{3.2.2 Residual contribution:} By unfolding the definition of Theorem \ref{thm:pretrace}, we obtain
$$I_\text{res}=\frac{1}{\sqrt{y_1y_2}}\frac{3}{\pi}\sum_{\chi^2=1}\chi(y_1)\overline{\chi(y_2)}\int_{\overline{G}(\mathbb{A})}f_{\mathbb{A}}(g)\chi(\det g)dg \int_{\mathbb{Q}\backslash \mathbb{A}}\overline{\theta_{m_1}(n_1)}dn_1 \int_{\mathbb{Q}\backslash \mathbb{A}}\theta_{m_2}(n_2)dn_2$$
and since $m_1, m_2$ are non-zero (their sign is irrelevant), the last two integrals identically vanish. \bigbreak 

\noindent \emph{3.2.3 Continuous contribution:} Recall the definition of the continuous family $\mathcal{F}_E(f)$ in \eqref{def:continuous-family},
an analogous computation as in the cuspidal term gives
\begin{equation}
    \int_{0}^{\infty}I_\text{cts}(w)dw=\frac{1}{8\pi}\sum_{\chi \in \mathcal{F}_E(f)}\sum_{\phi \in \mathcal{B}(\chi,\chi^{-1})}\int_{-\infty}^{\infty}h_\infty(t)a_{u_{E(\pi_{it}(\chi)(f)\phi_{it})}}
(m_1)\overline{a_{u_{E(\phi_{it})}}(m_2)}dt\end{equation}
Again if $m_1,m_2$ have opposite signs, the only change is in the form of the parity of the archimedean component of $\pi_{\chi|\cdot|^{it},\chi^{-1}|\cdot|^{-it}}$.
\bigbreak 

\noindent \emph{3.2.4 The geometric side}: Let $H=N\times N$, and let $(x,y)\in H(\mathbb{Q})$ act on $\gamma \in G(\mathbb{Q})$ via $(x,y)\cdot \gamma = x^{-1}\gamma y$. Given an orbit representative $\delta \in G(\mathbb{Q})$ for this action, we may define the orbital integrals $I_\delta(f_\mathbb{A})$ by
\begin{equation}\label{eq:orbit-geom-integral}
    I_\delta(f_\mathbb{A})=\frac{1}{\sqrt{y_1y_2}}\int_{H_\delta(\mathbb{Q})\backslash H(\mathbb{A})}f_\mathbb{A}(\begin{psmallmatrix}
        y_1 & x_1 \\ 0 & 1
    \end{psmallmatrix}^{-1}\delta\begin{psmallmatrix}
        y_2 & x_2 \\ 0 & 1
    \end{psmallmatrix})\overline{\theta_{m_1}(x_1)}\theta_{m_2}(x_2)dx_1 dx_2
\end{equation}
where $H_\delta(\mathbb{Q})$ is the stabiliser of $\delta$ in $H(\mathbb{Q})$. 

Following \cite{KL-petersson}, and the absolute convergence of $K_\text{geom}(x,y)$ on compacta, we may rewrite
\begin{equation}
    I = \sum_{[\delta]\ \text{relevant}} I_\delta(f_\mathbb{A})
\end{equation}
where the orbit $[\delta]$ is relevant if it has a representative in 
\begin{equation}\label{def:cells}
    \left\{ \begin{psmallmatrix}
        m_2/m_1 & 0 \\ 0 & 1
    \end{psmallmatrix}  \right\}\cup \left\{ \begin{psmallmatrix}
        0 & -\mu \\ 1 & 0
    \end{psmallmatrix} \bigr| \mu \in\mathbb{Q}^\times\right\}
\end{equation}
We call the term on the left  \emph{first cell term}, and the terms on the right the \emph{second cell terms}. We will not discuss the first cell term, since in the opposite sign case it provides no contribution by the support of $f_\infty$, but we remark that in the same sign case it provides the diagonal term. For details, see (\cite{KL-kuznetsov}, Proposition 7.9). We focus instead on the second cell terms. The orbital integrals factor in an archimedean and non-archimedean part, say $I_\delta(f_\mathbb{A})=I_\delta(f_\infty)I_\delta(f)$. We may write
\begin{equation}\label{eq:unrefined_geometric}
    I_\delta(f_\mathbb{A})=\frac{I_\delta(f)}{\sqrt{y_1y_2}}\iint_{\mathbb{R}^2}k(z_1,-\mu/z_2)e(m_2x_2-m_1x_1)dx_1dx_2
\end{equation}
where 
\begin{equation}\label{def:pair-point-kernel}
    k(z_1,z_2)=f_\infty(g_1^{-1}g_2) \qquad g_j\cdot i = z_j
\end{equation}
is a pair point invariant kernel depending only on $u(z_1,z_2):=|z_1-z_2|^2/(4 \ \text{Im}z_1\  \text{Im}z_2)$, so we shall often write $k(u(z_1,z_2))$ instead. We define the Zagier transform $\mathcal{Z}(t)$ of $k(z_1,z_2)$ by
\begin{equation}
    \mathcal{Z}(t)=\iint_{\mathbb{H}}k(z+t,-\frac{1}{z})\ y dz
\end{equation}
where $dz$ is the standard area element in the upper half plane. Imposing again the constraint $w=m_1y_1=m_2y_2$, evaluating the integral with respect to $w$ essentially reduces to computing the Fourier transform of $\mathcal{Z}$, that is
\begin{equation}
    \int_{0}^{\infty}I_\delta(f_\mathbb{A},w)dw=\sqrt{\mu m_1m_2} \ I_\delta(f)\int_{-\infty}^{\infty}\mathcal{Z}(t)e(-\sqrt{\mu m_1 m_2}\ t)dt  
\end{equation}
The computations can be found in \cite[\S 7.5]{KL-kuznetsov} and \cite[\S 2.2.7]{generalised}. As for $I_\delta(f)$, from the support of the function $f$ it follows that $\mu \in \mathbb{Z}_p^\times (\mathbb{Q}_p^\times)^2$. Since this holds for all $p$, the integral $I_\delta(f)=0$ unless there is $c \in \mathbb{Q_+}$ such that
\begin{equation}\label{eq:parametrisation}
    \mu = 1/c^2,
\end{equation}
in which case we recognise as the generalised Kloosterman sums $H(m_1,m_2,c)$ defined in \eqref{def:kloosterman-sum}.

We shall address the geometric side of the opposite sign case by introducing the modified Zagier transform
\begin{equation}\label{def:modified-zagier}
    \mathcal{M}(t)=\iint_{\mathbb{H}}k(z+t,\frac{1}{\overline{z}})ydz
\end{equation}
This transform shares several properties with $\mathcal{Z}(t)$, including being compactly supported and an even function. 

\subsection{Holomorphic variation}\label{section:holomrophic-variation} Fix $\kappa > 2$. Choosing $f_\infty$ as defined in \eqref{def:matrix-coeff}, we consider a variation of the $I$ given in \eqref{eq:decomposition}. More precisely, for
\begin{equation}
    I= \iint_{[N]^2}K_\text{geom}(n_1,n_2)\overline{\theta_{m_1}(n_1)}\theta_{m_2}(n_2)dn_1 dn_2
\end{equation} we obtain a simplified version of \eqref{eq:decomposition}, namely
\begin{equation}\label{eq:I-holomorphic}
    I = I_\text{cusp}
\end{equation}
where, in this setting,
\begin{align}
    I_{\operatorname{cusp}} :=\sum_{\pi \in \mathcal{F}_\kappa(f)}\sum_{\varphi \in \mathcal{B}(\pi)}W_{\pi(f)\varphi}(a(-m_1))\overline{W_\varphi(a(-m_2))} 
\end{align}
since the other terms vanish because $R(f_\mathbb{A})$ annihilates orthocomplement of the $L_0^2([G])$. 
We remark again that for $\kappa > 2$, the geometric kernel $K_\text{geom}(x,y)$ is still absolutely convergent, but this is not so for $\kappa=2$, which is the main difficulty in this problem. \bigbreak

Let $m_1,m_2 \in \tfrac{1}{N}\mathbb{N}$. Unfolding the definition of $I$ and using the relation \eqref{eq:whittaker-fourier} between the Whittaker and the Fourier coefficients, we obtain
\begin{equation}
    I_\text{cusp}=\frac{(4\pi)^\kappa}{\Gamma(k)}(m_1m_2)^{(\kappa-1)/2}e^{-2\pi(m_1+m_2)}\sum_{\pi \in \mathcal{F}_\kappa(\pi)}\sum_{\varphi \in \mathcal{B}(\pi)}a_{u_{\pi(f)\varphi}}(m_1)\overline{{a_{u_\varphi}(m_2)}}
\end{equation}
and vanishes otherwise. The geometric side admits the same decomposition in terms of orbital integrals $I_\delta(f_\mathbb{A})$. Using \cite[Propositions 3.4 and 3.6]{KL-petersson} for the archimedean aspect, if $m_1, m_2>0$ then
\begin{equation*}
    I_{\begin{psmallmatrix}
        m_2/m_1 & 0 \\ 0 & 1
    \end{psmallmatrix}}=\delta_{m_1=m_2 \in \tfrac{1}{N}\mathbb{N} }\frac{(4\pi \sqrt{m_1 m_2})^{\kappa-1}}{\Gamma(\kappa-1)} e^{-2\pi(m_1+m_2)}\int_{\mathbb{A}_\text{fin}}f(\begin{psmallmatrix}
        1 & t \\ 0 & 1
    \end{psmallmatrix})\theta_{\text{fin}}(-mt)dt
\end{equation*}
with  $m$ being the value of $m_1$ or $m_2$ in the case that $m_1=m_2$. Similarly
\begin{equation*}
    I_{\begin{psmallmatrix}
        0 & -\mu \\ 1 & 0
    \end{psmallmatrix}}(f_\mathbb{A})=\frac{(4\pi i)^{\kappa}(\sqrt{m_1 m_2})^{\kappa-1}e^{-2\pi(m_1+m_2)}}{2\Gamma(\kappa-1)}\frac{H(m_1,m_2,c)}{c} J_{\kappa-1}\left( \frac{4\pi\sqrt{m_1m_2}}{c} \right)
\end{equation*}
when $\mu = c^{-2}$ for $c \in \mathbb{Q}_+$, and vanishes otherwise. We may thus equate the geometric and spectral side to obtain the Petersson trace formula.

\section{Opposite signs}\label{section:opposite-signs}
In this section we prove both the refined and unrefined Bruggeman-Kuznetsov formulas, i.e Theorem \ref{thm:unrefinedBKformula} and Theorem \ref{thm:refined-opposite-signs} in the case of opposite sign case, as well as Corollary \ref{cor:equidistribution}. For purely notational practicality, we consider $m_1, m_2 \in \tfrac{1}{N}\mathbb{N}$, and compute 
\begin{equation}
    I = \frac{1}{\sqrt{y_1y_2}}\iint_{[N]^2}K_\text{geom}(n_1\begin{psmallmatrix}
        y_1 & 0 \\ 0 & 1
    \end{psmallmatrix},n_2\begin{psmallmatrix}
        y_2 & 0 \\ 0 & 1
    \end{psmallmatrix})\overline{\theta_{-m_1}(n_1)}\theta_{m_2}(n_2) dn_1dn_2
\end{equation}
in the two ways given by the pre-trace formula \eqref{eq:I-holomorphic}. Once again $K_{\operatorname{geom}}$ is the kernel attached to the test function $f_\infty\otimes f$ where $f_\infty \in C_c^\infty(G^+(\mathbb{R})//K_\infty)$ and $f \in \mathcal{H}$ satisfies the conditions of Theorem \ref{thm:unrefinedBKformula} (and later of Theorem \ref{thm:refined-opposite-signs}). The discussion of the spectral side follows very closely to that in Section \ref{section:same-sign}, and thus we shall simply recall those results. Likewise, we quickly deal with the first cell terms. The bulk of the analysis will focus of expressing the Fourier transform of $\mathcal{M}(t)$, defined in \eqref{def:modified-zagier}, in terms of $h_\infty$ to recover $H_\infty^-$. We start with the case that $h_\infty$ is of Paley-Wiener type, and later appeal to \cite[\S 8]{KL-kuznetsov} for extending the class of test functions to \eqref{def:testfunction-conditions}. Of course the case with $m_1>0$ and $m_2<0$ is derived exactly the same way.

\subsection{Overview of the Spectral side}\label{section:spectral-overview-opposite-signs}
\emph{4.1.1 Cuspidal contribution}: With the same notation as in \eqref{eq:decomposition}, given $y_1,y_2>0$,
\begin{align}\label{eq:cusp-unrefined-oppositesigns}
        I_\text{cusp}=\frac{4}{\pi}\sum_{\pi \in \mathcal{F}_0(f)}p(\pi)h_\infty(t_\pi)(\cosh (\pi t_\pi))& K_{it_\pi}(2\pi m_1y_1)K_{it_\pi}(2\pi m_2 y_2)\notag  \\
        & \times\sum_{\varphi \in \mathcal{B}(\pi)}a_{u_{\pi(f)\varphi}}(m_1)\overline{a_{u_\varphi}(m_2)}  
\end{align}
where the factor $p(\pi)$ arises from the identity $a_{u_{\pi(f)\varphi}}(-m_1)=p(\pi)a_{u_{\pi(f)\varphi}}(m_1)$.
Setting $w=m_1y_1=m_2y_2$ and denoting $I_\text{cusp}(w)$ the formula with this constraint becomes
\begin{equation}\label{eq:opposite-sign-cusp-refined}
    \int_{0}^{\infty}I_\text{cusp}(w)dw=\frac{1}{2}\sum_{\pi \in \mathcal{F}_0(f)}p(\pi)h_\infty(t_\pi)\sum_{\varphi \in\mathcal{B}(\pi)}a_{u_{\pi(f)\varphi}}(m_1)\overline{a_{u_{\varphi}}(m_2)}
\end{equation} \bigbreak

\noindent \emph{4.1.2 Continuous contribution}: By an analogous argument, using the same constraint in terms of $w$, the unrefined and refined continuous contributions are, respectively,
\begin{align}\label{eq:unrefined-cts}
    I_\text{cts}=\frac{1}{4\pi}\sum_{\chi \in \mathcal{F}_E(f)}p(\pi_{\chi,\chi^{-1}})& \sum_{\phi \in \mathcal{B}(\chi,\chi^{-1})} \int_{-\infty}^{\infty}{h}_\infty
(t)\cosh(\pi t)\notag \\ &\times K_{it}(2\pi m_1y_1)K_{it}(2\pi m_2y_2)a_{E(\pi_{it}(\chi)(f)\phi_{it})}(m_1)\overline{a_{E(\phi_{it})}(m_2)}dt\end{align}
and
\begin{equation}\label{eq:opposite-sign-cts-refined}
    \int_{0}^{\infty}I_{\text{cts}}(w)dw = \frac{1}{8\pi}\sum_{\chi \in \mathcal{F}_E(f)}p(\pi_{\chi,\chi^{-1}})\int_{-\infty}^{\infty}h_\infty(t)w(\pi_{it}(\chi),f)\lambda_{\pi_{it}(\chi)}(m_1)\overline{\lambda_{\pi_{it}(\chi)}(m_2)}dt
\end{equation} \bigbreak

\noindent \emph{4.1.3 Residual contribution}: We may rewrite $I_\text{res}$ as in Section 3.2.2, with the only difference being that $m_1$ is replaced with $-m_1$. Since $-m_1,m_2$ are non-zero, the integral vanishes identically. \bigbreak

We remark that all the sums and integrals in Section \ref{section:spectral-overview-opposite-signs} converge absolutely. In the cases \eqref{eq:cusp-unrefined-oppositesigns} and \eqref{eq:unrefined-cts}, this is simply a consequence of Theorem \ref{thm:pretrace}, and in the refined cases \eqref{eq:opposite-sign-cusp-refined} and \eqref{eq:opposite-sign-cts-refined}, we refer to \cite[\S 7.1 to 7.4]{KL-kuznetsov}, applying suitable changes where necessary.

\subsection{The modified Zagier transform}
Our goal in this section is proving the following proposition concerning the modified Zagier transform

\begin{proposition}\label{prop:DtoHtransform}
    Let $h_\infty\in PW^{\operatorname{even}}(\mathbb{C})$ related to $f_\infty \in C_c^\infty(G^+(\mathbb{R})//K_\infty)$ via Proposition \ref{prop:paley-wiener-isom}. For $a>0$ we have
    \begin{equation}
        \widehat{\mathcal{M}}(a)=\frac{1}{2\pi a}\int_{-\infty}^{\infty}K_{2it}(4\pi a)\sinh(\pi t)t \ h_\infty(t)dt
    \end{equation}
    where $\widehat{\mathcal{M}}$ denotes the Fourier transform of $\mathcal{M}$ given in \eqref{def:modified-zagier}.
\end{proposition}
We first introduce a couple of lemmas in order to express $\mathcal{M}$ in terms of $h_\infty$. The expression thus obtained cannot be used directly due to issues of convergence, therefore we shall prove the proposition by performing a countour shift.

\begin{lemma}\label{lemma:D-to-kernel}
    Let $ h_\infty$ and $f_\infty$ be defined as in Proposition \ref{prop:DtoHtransform}, then for $t \in \mathbb{R}$, we have
    \begin{equation}
        \mathcal{M}(t)=\pi\sqrt{1+t^2}\int_{-\infty}^{\infty}k(\left(\frac{t^2}{4}+1\right)x^2+\frac{t^2}{4})dx
    \end{equation}
    where $k(u)$ is the pair-point invariant kernel attached to $f_\infty$ given in \eqref{def:pair-point-kernel}.
\end{lemma}

\begin{proof}
    Let $a=t/2+\sqrt{t^2/4+1}$ and $b=t/2-\sqrt{t^2/4+1}$, and apply the change of variables $$w=\zeta+i\eta := \frac{z+b}{z+a} \qquad \Longrightarrow \qquad z=\frac{aw-b}{-w+1}$$
    which is a conformal map $\mathbb{H} \to \mathbb{H}$. Since $\overline{z}^{-1}=z|\overline{z}|^{-2}$, we have
    $$u(z+t, \frac{1}{\overline{z}})=\frac{|z\overline{z}+t\overline{z}-1|^2}{4\text{Im}(z)^2}$$
    The imaginary parts of $z$ and $w$ are further related by $\text{Im}(z)=\sqrt{t^2+4}\ \eta/|-w+1|^2$, so $u$ may be further simplified as $|T(w)|^2/(4(t^2+4)\eta^2)$ where
    \begin{equation}
    T(w)=|-w+1|^2\left|\frac{(aw-b)(a\overline{w}-b)}{|-w+1|^2}+t\frac{a\overline{w}-b}{-\overline{w}+1}-1\right|
    \end{equation}
    Upon collecting terms, noting that $a,b$ are roots of $X^2-tX-1$
    $$T(w)=(-ab+bt+1)w+(-ab+at+1)\overline{w}$$ thus in terms of $\zeta,\eta$ we obtain
    \begin{equation}
    \mathcal{M}(t)=\sqrt{t^2+4}\int_{\mathbb{H}}k\left((\frac{t^2}{4}+1)\frac{\zeta^2}{\eta^2}+\frac{t^2}{4}\right)\frac{1}{(1-\zeta)^2+\eta^2}\frac{d\eta d\zeta}{\eta}
    \end{equation}
    Changing coordinates with $x=\zeta/\eta$ and $y=\zeta$ then
    $$\mathcal{M}(t)=\sqrt{t^2+4}\int_{0}^{\infty}k\left((\frac{t^2}{4}+1)x^2+\frac{t^2}{4}\right)\int_{-\infty}^{\infty}\frac{dy}{(1-y)^2+y^2/x^2}\frac{dx}{x}$$
    We conclude by noting that the inner integral is $\pi x$.
\end{proof}

As immediate consequences of the lemma, $\mathcal{M}(t)$ is even and compactly supported. We now proceed in obtaining an explicit relation with $h_\infty$.

\begin{lemma}
    Let $h_\infty$ and $f_\infty$ be defined as in Proposition \ref{prop:DtoHtransform}, then
    \begin{equation}\label{eq:D-to-h-infty}
        \mathcal{M}(t)=\frac{1}{4}\int_{-\infty}^{\infty}h_\infty(r)e^{i\alpha r}dr
    \end{equation}
        where $\alpha$ satisfies $t=2\sinh(\alpha/2)$.
\end{lemma}

\begin{proof}
    Letting $t=2\sinh(\alpha/2)$ and $v=x^2+\sinh^2(\alpha/2)$ in the expression in Lemma \ref{lemma:D-to-kernel}, we may rewrite $\mathcal{M}(t)$ as
    \begin{equation}
        \mathcal{M}(t)=\pi\int_{\sinh^2(\alpha/2)}^{\infty}k(v)\frac{1}{\sqrt{v-\sinh^2(\alpha/2)}}dv
    \end{equation}
        By the inversion formulas (1.62) and (1.64) in \cite{Spectralmethods} the integral above can be expressed as
        $$\frac{1}{4}\int_{-\infty}^{\infty}h(r)e^{i\alpha r}dr$$
        which completes the proof.
\end{proof}

\begin{proof}[Proof of Proposition \ref{prop:DtoHtransform}] Using directly the integral expression \eqref{eq:D-to-h-infty} and exchanging the order of integration is not possible because of failure of absolute convergence. We evaluate the integral by performing a contour shift.
    For $\epsilon>0$ and $x>0$, consider 
    $$ I_+(x,\epsilon;r)  =\int_{0}^{\infty}e^{(ix-\epsilon)\sinh\alpha+2ir\alpha}\cosh\alpha d\alpha \qquad I_-(x,\epsilon;r)  =\int_{0}^{\infty}e^{(-ix-\epsilon)\sinh\alpha-2ir\alpha}\cosh\alpha d\alpha$$
    both of which are absolutely convergent by the exponential decay of the integrand. Upon integrating by parts
    $$I_+(x,\epsilon;r)=\frac{-1}{ix-\epsilon}-\frac{2ir}{ix-\epsilon}\int_{0}^{\infty}e^{(ix-\epsilon)\sinh\alpha+2ir\alpha}d\alpha$$
    Consider the rectangular contour coming from $\infty+i\pi/2$ to $i\pi/2$, going down from $i\pi/2$ to $0$ and then going back to $\infty$, by Cauchy theorem
    \begin{equation}\label{eq:bigIplus}
    \begin{split}
        I_+(x,\epsilon;r) = & \frac{-1}{ix-\epsilon}-\frac{2ir}{ix-\epsilon}\int_{0}^{i\pi/2}e^{(ix-\epsilon)\sinh\alpha+2ir\alpha}d\alpha \\
        &-\frac{2ir}{ix-\epsilon}e^{-\pi r}\int_{0}^{\infty}e^{(-x-i\epsilon)\cosh\alpha+2ir\alpha}d\alpha
    \end{split}    
    \end{equation}
    Similarly, considering the rectangular contour from $\infty-i\pi/2$ to $-i\pi/2$, going up from $-i\pi/2$ to $0$, and then going back to $\infty$
    \begin{equation}\label{eq:bigIminus}
        \begin{split}
        I_-(x,\epsilon;r) = &  \frac{1}{ix+\epsilon}-\frac{2ir}{ix+\epsilon}\int_{0}^{-i\pi/2}e^{(-ix-\epsilon)\sinh\alpha-2ir\alpha}d\alpha \\
        & -\frac{2ir}{ix+\epsilon}e^{-\pi r}\int_{0}^{\infty}e^{(-x+i\epsilon)\cosh\alpha-2ir\alpha}d\alpha
        \end{split}
    \end{equation}
    By the dominated convergence theorem and the $K$-Bessel expression  \cite[(8.432.1)]{Gradshteyn}, we have 
    $$\lim_{\epsilon \to 0} \ (I_+(x,\epsilon;r)+I_-(x,\epsilon;r))=-\frac{4r}{x}e^{-\pi r}\int_{0}^{\infty}e^{-x\cosh\alpha}\cosh(2ir\alpha)d\alpha=-\frac{4r}{x}e^{-\pi r}K_{2ir}(x)$$
    By the compact support of $\mathcal{M}$, we have
    \begin{equation*}
        \begin{split}
            \widehat{\mathcal{M}}(a) & = \lim_{\epsilon \to 0}\left(\int_{0}^{\infty}\mathcal{M}(t)e^{2\pi (ia-\epsilon)t}dt+\int_{0}^{\infty}\mathcal{M}(-t)e^{-2\pi(ia+\epsilon)t}dt\right) = \\
            & =\frac{1}{4} \ \lim_{\epsilon \to 0} \int_{0}^{\infty}\left( \int_{-\infty}^{\infty}h_\infty(r)e^{i\alpha r+2\pi (ia-\epsilon)t}dr+\int_{-\infty}^{\infty}h_\infty(r)e^{-i\alpha r-2\pi(ia+\epsilon)t}dr\right)dt
        \end{split}
    \end{equation*}
    where we recall $t=2\sinh(\alpha/2)$. Since $h_\infty(r) \in PW^{\operatorname{even}}(\mathbb{C})$ and by the exponential decay in $t$, we may apply Fubini to changing the order of integration. Applying changes of variables $t\mapsto2\sinh(\alpha/2)$ and $\alpha \mapsto 2\alpha$
    $$\widehat{\mathcal{M}}(a)=\frac{1}{2}\ \lim_{\epsilon \to 0}\int_{-\infty}^{\infty}h_\infty(r)(I_+(4\pi a,\epsilon;r)+I_-(4\pi a,\epsilon;r))dr$$
   Using again the polynomial decay of $h_\infty$ and  \eqref{eq:bigIplus}, \eqref{eq:bigIminus}, we deduce by dominating convergence
   $$\widehat{\mathcal{M}}(a)=-\frac{1}{2\pi a}\int_{-\infty}^{\infty}K_{2ir}(4\pi a)e^{-\pi r }rh_\infty (r)dr = \frac{1}{2\pi a}\int_{-\infty}^{\infty}K_{2ir}(4\pi a)\sinh (\pi r)r h_\infty(r)dr$$
   where the second equality is obtained upon splitting the integrand in terms of the even and odd parts. The proposition is then proven up to relabeling variables.
\end{proof}

\subsection{Proof of Theorems \ref{thm:unrefinedBKformula} and \ref{thm:refined-opposite-signs}}
 Given the results for the cuspidal and continuous contribution from Section 4.1, to prove Theorem \ref{thm:unrefinedBKformula} and Theorem \ref{thm:refined-opposite-signs} it suffices to obtain an expression for $I_\delta(f_{\mathbb{A}})$ where $\delta$ is a first or second cell terms given in \eqref{def:cells}. \bigbreak 

\noindent \emph{4.3.1 First cell term}: Considering the archimedean component
\begin{equation}
    I_{\big(\begin{smallmatrix}
    -m_2/m_1 & \\ & 1
\end{smallmatrix}\ \big)}(f_\infty)=\frac{1}{\sqrt{y_1y_2}}\int_{-\infty}^{\infty}f_\infty(\begin{psmallmatrix}
    m_2y_2  & t \\ & -m_1y_2
\end{psmallmatrix})e(t)dt
\end{equation}
it is immediate by the determinant condition on the support of $f_\infty$ that the first cell term contribution identically vanishes. \bigbreak 

\noindent \emph{4.3.2 Second cell terms}: Since $f_\infty$ is supported on $G^+(\mathbb{R})$ and it is bi-$K_\infty$-invariant, for $\delta=\begin{psmallmatrix}
    0 & -\mu \\ 1 & 0
\end{psmallmatrix}$ with $\mu \in \mathbb{Q}^\times$, we may follow \cite[\S 7.5.2]{KL-kuznetsov} to obtain
\begin{equation}\label{eq:unrefined-orbits}
    I_\delta(f_\mathbb{A})=\frac{I_\delta(f)}{\sqrt{y_1y_2}}\iint_{\mathbb{R}^2}k(z_1,\frac{-\mu}{z_2})e(m_2x_2+m_1x_1)dx_2dx_1
\end{equation}
when $\mu = 1/c^2$ and it is zero otherwise.
Let $x_1=\sqrt{\mu m_2/m_1}t_1$ and $x_2=\sqrt{\mu m_1/m_2}t_2$ so that $dx_1dx_2=\mu dt_1dt_2$ and 
$$\big(x_1+iy_1,\frac{-\mu}{x_2+iy_2} \big)= \sqrt{\frac{\mu m_2}{m_1}}\big( t_1+i\frac{m_1y_1}{\sqrt{\mu m_1m_2}},\frac{-1}{t_2+i\frac{m_2y_2}{\sqrt{\mu m_1m_2}}}\big)$$
Setting $w = m_1y_1=m_2y_2$, the integral of $I_\delta(f_\mathbb{A},w)$ is the product of $\mu\sqrt{m_1m_2}\ I_\delta(f)$ and 
$$\int_{0}^{\infty}\iint_{\mathbb{R}^2}k\big(t_1+i\frac{w}{\sqrt{\mu m_1m_2}},\frac{-1}{t_2+i\frac{w}{\sqrt{\mu m_1m_2}}}\big)e( \sqrt{\mu m_1m_2}(t_1+t_2))dt_1dt_2\frac{dw}{w}$$
We now perform another change of variables, namely $t=t_1+t_2$, $x=-t_2$ and $y=w/\sqrt{\mu m_1m_2}$ to get
$$\int_{0}^{\infty}\int_{-\infty}^{\infty}\int_{-\infty}^{\infty}k(x+t+iy,\frac{1}{x-iy})e( \sqrt{\mu m_1m_2}t)dtdx\frac{dy}{y}$$
Exchanging the order of integration, which is justified by the compact support of $\mathcal{M}(t)$ (see comment after Lemma \ref{lemma:D-to-kernel}), we obtain
\begin{equation}\label{eq:opposite-sign-geom-side}
\int_0^{\infty}I_\delta(f,w)dw=\sqrt{\mu m_1m_2}I_\delta(f)\int_{-\infty}^{\infty}\mathcal{M}(t)e( \sqrt{\mu m_1m_2}t)dt
\end{equation}
We now present the proof of Theorem \ref{thm:unrefinedBKformula} and Theorem \ref{thm:refined-opposite-signs}, though this is really just a recollection of the formulas that we have proved so far.
\begin{proof}[Proof of Theorem \ref{thm:unrefinedBKformula} and Theorem \ref{thm:refined-opposite-signs} ] The first theorem comes from combining equations  \eqref{eq:opposite-sign-cusp-refined}, \eqref{eq:opposite-sign-cts-refined}, \eqref{eq:opposite-sign-geom-side} together with Proposition \ref{prop:DtoHtransform}. There is a delicate discussion sorrounding justifying exchanging summation and integration in
$$\int_0^\infty\sum_{c\in\mathcal{C}(f)}I_\delta(f_\mathbb{A},w)dw=\int_0^\infty\sum_{c\in \mathcal{C}(f)}H(m,n,c)I_\delta(f_\infty,w)dw$$
Looking back to equation \eqref{eq:unrefined-orbits} and performing change of variable $wc/\sqrt{m_1m_2} \mapsto w$, it suffices to show that
$$\iint_{\mathbb{R}^2}\left|k\left(t_1+iw, \frac{-1}{t_2+iw}\right) \right|dt_1 dt_2 $$
vanishes identically for all sufficiently large or small values of $w$, uniformly in $c$. This is shown in \cite[\S 2.2.7]{generalised}.

Now suppose that $f$ satisfies the geometric and spectral assumptions, then by Proposition \ref{prop:kloosterman-properties}, we may restrict our attention to $m_1,m_2 \in \mathbb{Z}$. Making use of equations \eqref{def:discrete-weights} and \eqref{def:continuous-weights} we may rewrite the spectral side as
\begin{gather}
    \frac{1}{2}\sum_{\pi \in \mathcal{F}_0(f)}p(\pi)h_\infty(t_\pi)w(\pi,f)\lambda_{\pi}(m_1)\overline{\lambda_\pi(m_2)} \notag \\
    + \frac{1}{8\pi}\sum_{\chi \in \mathcal{F}_E(f)}p(\pi_{\chi,\chi^{-1}})\int_{-\infty}^{\infty}h_\infty(t)w(\pi_{it}(\chi),f)\lambda_{\pi_{it}(\chi)}(m_1)\overline{\lambda_{\pi_{it}(\chi)}(m_2)}dt \label{eq:analytic-proof}
\end{gather}
    The refined geometric side is instead
    \begin{equation}\label{eq:geometric-proof}
        \frac{1}{2}\sum_{c \in \mathcal{C}(\mathcal{F})}\frac{H(m_1,m_2,c)}{c}H_\infty^{-}\left(\frac{4\pi\sqrt{m_1m_2}}{c} \right)
    \end{equation}
    Equating \eqref{eq:analytic-proof} and \eqref{eq:geometric-proof}, and multiplying both sides completes the proof in the case $h_\infty \in PW^{\operatorname{even}}(\mathbb{C})$. 
    Lastly, we extend the proof to the larger class of $h_\infty$ satisfying \eqref{def:testfunction-conditions}. We begin by noting that, following \cite[Proposition 8.19]{KL-kuznetsov}, we may extend the map given in Proposition \ref{prop:paley-wiener-isom} to a map from some subset of $S\subset C(G^+(\mathbb{R})//K)$ to the space of functions satisfying \eqref{def:testfunction-conditions}. All intermediate Lemmas towards proving Proposition \ref{prop:DtoHtransform} (and thus the proposition itself) are then still valid for the larger class of test functions.
    From \cite[Theorem 3.8 (3)]{generalised}, we may factor the generalised Kloosterman sums as
    \begin{equation}\label{eq:kloosterman-factorisation}
        H(m,n,c)=S(\overline{c_N}m,\overline{c_N}n,c_0)H(m\overline{c_0},n\overline{c_0},c_N)
    \end{equation}
    where $N$ is the level of $f$, and $c=c_0c_N$ with $c_0,c_N \in \mathbb{N}$ and $(c_0,N)=1$. The absolute convergence of \eqref{eq:geometric-proof} is obtained by combining the Weyl bound, the trivial bound from Proposition \ref{prop:kloosterman-properties}, and the bound on $H_\infty$ from Lemma \ref{lemma:H-minus-bound}.
    Then using \cite[Propositions 8.24 and 8.25]{KL-kuznetsov}, the spectral side converges absolutely, and it equals the geometric side.
\end{proof}

\subsection{Parity equidistribution}\label{section:corollary}
We conclude the chapter by proving the parity equidistribution result. We take $h_\infty$ to be one of the two test functions given in \eqref{def:testfunction1} or \eqref{def:testfunction2}. We obtain our corollary by combining two distribution results, obtained by the same and opposite sign case BK formulas respectively.
\begin{proposition}\label{prop:positive-bound}
    Let $h_\infty$ be one of \eqref{def:testfunction1} or \eqref{def:testfunction2} , if $f \in \mathcal{H}$ satisfies both the Geometric and Spectral assumptions, then for all $m_1,m_2 >0$ with $(m_1m_2,N)=1$ we have
    \begin{equation}
        \sum_{\pi \in \mathcal{F}_0(f)}h_\infty(t_\pi)w(\pi,f)\lambda_{\pi}(m_1)\overline{\lambda_{\pi}(m_2)}+(cts.)=\delta_{m_1=m_2}\delta+O\left(\frac{f_{\mathbb{A}}(1)m_1m_2}{T^2k(\mathcal{F})}\right)
    \end{equation}
\end{proposition}
\begin{proof}
    See \cite[Lemma 1.13]{generalised}.
\end{proof}
We will use the BK formula for opposite signs to obtain a similar formula that features the parity. Since we want to obtain bounds on the geometric side, and we know the trivial bound on $H(m,n,c)$ from \eqref{eq:trivial-bound}, it suffices to prove a bound on $H_\infty^-(x)$.
\begin{lemma} \label{lemma:H-minus-bound}
For $h_\infty$ one of \eqref{def:testfunction1} or \eqref{def:testfunction2}, its transform $H_\infty^-$ satisfies the inequality
\begin{equation}
    H_\infty^-(x)\ll f_\infty(1)\left(\frac{x}{T} \right)^2\exp(4\pi x)
\end{equation}
\end{lemma}
\begin{proof}
    Using the relation between the $K$-Bessel and $I$-Bessel function
    \begin{equation*}
        K_{2it}(x)=\frac{\pi}{2i \sinh(2\pi t)}(I_{-2it}(x)-I_{2it}(x))
    \end{equation*}
    (see e.g. \cite[(B.34)]{Spectralmethods}) and the double angle formula for hyperbolic functions, we may rewrite the transform $H_\infty^-$ as
    \begin{equation}
        H_\infty^-(x)=\frac{i}{2}\int_{-\infty}^{\infty}I_{2it}(x)\frac{th_\infty(t)}{\cosh (\pi t)}dt
    \end{equation}
    Shifting the contour to $\text{Im}(z)=-1$, we have
    \begin{equation}
        H_\infty^-(x) \ll \int_{-\infty}^{\infty}\frac{(|t|+1)h_\infty (t)}{\cosh(\pi t)}|I_{2+2it}(x)|dt \ll f_\infty(1)\left(\frac{x}{T}\right)^2 \exp(4\pi x)
    \end{equation}
    via the identity (cf \cite[(8.431.2)]{Gradshteyn} ), valid for $\text{Re}(\nu)>-1/2$,
    $$I_\nu(z)=\frac{(z/2)^\nu}{\Gamma(\nu+1/2)\Gamma(1/2)}\int_{-1}^{1}(1-t^2)^{\nu-1/2}\cosh(zt)dt.$$
\end{proof}
\begin{proof}[Proof of Corollary \ref{cor:equidistribution}]
    We apply Proposition \ref{prop:kloosterman-properties} \emph{3} and Lemma \ref{lemma:H-minus-bound} to the sum of Kloosterman sums to obtain
    \begin{equation}
    \begin{split}
        \sum_{c\in \mathcal{C}(\mathcal{F})}\frac{H(m_1,m_2,c)}{c}H_\infty^- \left( \frac{4\pi \sqrt{m_1m_2}}{c}\right) \ll \frac{f_\mathbb{A}(1)m_1m_2}{T^2} & \sum_{c\in \mathcal{C}(\mathcal{F})} \frac{1}{c^2} \ \exp\left(\frac{16\pi^2 \sqrt{m_1 m_2}}{c}\right)\\
        \ll \frac{f_\mathbb{A}(1)m_1m_2}{T^2 k(\mathcal{F})^2}\exp\left(\frac{16\pi^2\sqrt{m_1m_2}}{k(\mathcal{F})}\right)
    \end{split}
    \end{equation}
    using again \eqref{eq:trivial-bound}. We have thus proved that
    \begin{equation}\label{eq:negativebound}
    \sum_{\pi \in \mathcal{F}_0(f)}p(\pi)h_\infty(t_\pi)w(\pi,f)\lambda_\pi(m_1)\overline{\lambda_\pi(m_2)}+(cts.)=O\left(\frac{f_\mathbb{A}(1)m_1m_2}{T^2 \kappa^2(\mathcal{F})}\exp\left(\frac{16\pi^2\sqrt{m_1m_2}}{\kappa(\mathcal{F})}\right) \right)
    \end{equation}
    Upon taking $m_1=m_2=m$, adding and subtracting Proposition \ref{prop:positive-bound} and \eqref{eq:negativebound}, depending on the desired parity, gives the result.
\end{proof}
 
\section{The holomorphic case}\label{section:holomorphic}
We devote this section to the proof of the Petersson formula for $\kappa=2$, ie Theorem \ref{thm:p2-formula}. We assume throughout that $f \in \mathcal{H}$ is a level $N$ function that satisfies the Geometric and Spectral assumptions. We denote by $\omega$ the character of $K_\infty$ given by $k_\theta \mapsto e^{2i\theta}$ .

As in the previous chapter, we follow the plan of Section \ref{section:holomrophic-variation}. Since $f_\infty$ (see \eqref{def:matrix-coeff}) does not have an absolutely convergent geometric kernel, we introduce a smooth, bi-$\omega$-isotypic truncation. The orthocomplement of the $\omega$-isotypic cuspidal space need not be annihilated, thus we need to estimate the contribution on the spectral side.
Fix $\rho: \mathbb{R} \to [0,1]$ to be a smooth function such that
\begin{equation}
    \rho(x)=\begin{cases}
    1 \quad \text{if} \ x\leq 0 \\
    0 \quad \text{if} \ x\geq 1
\end{cases}
\end{equation}
For $T\geq 1$, let $\varrho^T: G(\mathbb{R}) \to \mathbb{C}$ be the function supported on $G^+(\mathbb{R})$ defined by the equirement that
\begin{equation}\label{eq:smooth-truncation}
    \varrho^T\left(\ k_{\theta_1}^{-1}\begin{psmallmatrix}
    e^{-r/2} & 0 \\ 0 & e^{r/2}
\end{psmallmatrix} k_{\theta_2} \right)=\rho\left(r-T\right) \qquad (\theta_1,\theta_2 \in [0,2\pi), r\geq 0).
\end{equation}
Note that the restriction $r\geq 0$ makes $\varrho^T$ well defined (see eg the proof of \cite[Proposition 6.13]{Traces}). We also remark that $\varrho^T$ is a compactly supported, smooth, bi-$K_\infty$-invariant function, and that for $g \in G(\mathbb{R})^+$ it depends only on the hyperbolic distance between the points $g\cdot i$ and $i$.

Defining 
\begin{equation}\label{def:truncated-test-function}
    f_\infty^T:=f_\infty \cdot \varrho^T
\end{equation}
it has the same properties as $\varrho^T$, with the only exception being that it is bi-$\omega$-isotypic rather than bi-$K_\infty$-invariant. We apply the adelic pre-trace formula in Theorem \ref{thm:pretrace} for $f_\mathbb{A}^T=f_\infty^T \otimes f$, where $f \in \mathcal{H}$ satisfies the geometric and spectral assumptions. Letting $I(T)$ be the integral on either side of \eqref{eq:I-holomorphic} when the archimedean test function $f_\infty^T$, our goal is to evaluate 
\begin{equation*}
    \lim_{T \to \infty} I(T)
\end{equation*}
in two ways, given as before by the decomposition \eqref{eq:decomposition}.

\subsection{The image of the truncated operator}
In this section we describe the subspace of $L^2([G])$ that is not annihilated by $R(f_\mathbb{A}^T)$. Let $\mathcal{F}(T,f)$ the family of standard cuspidal automorphic representations $\pi$ for which $\pi(f_\mathbb{A}^T)$ is not the zero operator. If $\pi \in \mathcal{F}(T,f)$, it is clear that $\pi_\infty$ is either a principal series or the discrete series $\mathcal{D}(2)$. Thus we may rewrite $\mathcal{F}(T,f)$ 
as a union
\begin{equation}
\mathcal{F}(T,f) = \mathcal{F}_0(T,f)\cup\mathcal{F}_2(T,f)
\end{equation}
depending on the isomorphism class of $\pi_\infty$. We remark that $\mathcal{F}_2(T,f)$ is a finite family. We similarly define $\mathcal{F}_E(T,f)$, the set of characters in $(\mathbb{Q}^\times \backslash \mathbb{A}^1)^\wedge$ for which $\pi(f)$ is non-zero given $\pi\simeq\pi_{it}(\chi)$ for some $t \in \mathbb{R}$. For $\pi$ in $\mathcal{F}(T,f)$ or $\mathcal{F}_E(T,f)$, the operator $\pi(f_\mathbb{A}^T)$ acts by a scalar on $\pi^{\omega\times K_0(N)}$ and annihilates its orthocomplement. For $\varphi \in \pi^{\omega \times K_0(N)}$ with $\pi$ in either $\mathcal{F}(T,f)$ or $\mathcal{F}_E(f,T)$ we write
\begin{equation}\label{eq:T-eigenvalue}
    \pi(f_\mathbb{A}^T)\varphi = \lambda_T(\varphi) \varphi
\end{equation}

\noindent \emph{5.1.1 Eigenvalues for the discrete series}
Recall the subspace $V_f \subseteq V_\pi$ given in \eqref{eq:V_f}.
\begin{proposition}\label{prop:discrete-limit}
    Suppose $\pi \in \mathcal{F}_2(T,f)$ and $\varphi$ lies in the $\omega$-isotypic subspace of $V_f$, then
    \begin{equation}
        \lim_{T \to \infty}\lambda_T(\varphi)=1
    \end{equation}
\end{proposition}
\begin{proof}
    Let $v_0$ a normalised lowest weight vector for $\mathcal{D}(2)$ , then we have
    \begin{equation}
        \lambda_T(\varphi)=\langle \pi_\infty(f_\infty^T)v_0, v_0 \rangle= \int_{\overline{G}(\mathbb{R})}f_\infty^T(g)\langle \pi_\infty(g)v_0,v_0\rangle dg=||\Phi_{v_0,2}||_{L^2(G(\mathbb{R}))}^2\int_{\overline{G}(\mathbb{R})}f_\infty^T(g) \overline{f_\infty(g)}dg
    \end{equation}
    Because $\pi_\infty$ is square integrable and $|f_\infty^T(g)|\leq |f_\infty(g)|$, it is possible to interchange the limit and the integral via the dominated convergence theorem, obtaining the result.
\end{proof}

\noindent \emph{5.1.2: Eigenvalues for principal series}:
Let $\pi_\infty = \pi(\text{sgn}^\epsilon |\cdot |^{it},\text{sgn}^\epsilon |\cdot |^{-it})$ with $t \in \mathbb{R\cup} (-1/2,1/2)i$, we work in the induced model described at the beginning of Section 2.1.1, with $\phi_2$ the weight $2$ vector in $\pi$ normalised so that $\phi_2(1)=1$. As mentioned in Section 1.4.3, the measures are as in \cite[\S 7]{Traces}. The scalar $\lambda_T(\phi_2)$ can be computed as follows:
\begin{equation}\label{eq:expression-lambda-T-principal}
    \begin{split}
        \lambda_{T} (\phi_2)& = \pi_\infty(f_\infty^T)\phi_2(1)=\int_{G(\mathbb{R})}f_\infty^T(g)\pi_\infty(g)\phi_2(1)dg =\\
        & =\int_{0}^\infty \int_{-\infty}^{\infty}\int_0^{2\pi}f_\infty^T\left(\begin{pmatrix}
            y^{1/2} & xy^{-1/2} \\ 0 & y^{-1/2}
        \end{pmatrix}k_\theta\right)\phi_2\left(\begin{pmatrix}
            y^{1/2} & xy^{-1/2} \\ 0 & y^{-1/2}
        \end{pmatrix}k_\theta)\right )\frac{d\theta}{2\pi}\frac{dxdy}{y^2}=\\
        & =\int_{0}^\infty \int_{-\infty}^{\infty}f_\infty^T\left(\begin{pmatrix}
            y^{1/2} & xy^{-1/2} \\ 0 & y^{-1/2}
        \end{pmatrix} \right)y^{it+1/2}\frac{dxdy}{y^2}
    \end{split}
\end{equation}
using that $\phi_2$ and $f_\infty$ have opposite weight on the right.
By above $\lambda_T(\phi_2)$ depends on $\phi_2$ only via the spectral parameter $t$, therefore we rewrite for convenience $\lambda_{T}(t):=\lambda_T(\phi_2)$.

\bigbreak For $z \in \mathbb{H}$, let $r(z)$ denote the hyperbolic distance between $i$ and $z$, and set
$$r_T(z)=r(z)-T$$
Note that $r(z)$ is given by
$$r(z) = \log\frac{|z+i|+|z-i|}{|z+i|-|z-i|}$$
By construction of $\varrho^T$ (see \eqref{eq:smooth-truncation}), then the integral can be rewritten as 
\begin{equation}\label{eq:eigenvalue}
    \lambda_{T}(t)=\frac{\kappa-1}{4\pi}(2i)^\kappa\int_{0}^{\infty}\int_{-\infty}^{\infty}\frac{\rho(r_T(z))}{(x+i(y+1))^2}y^{it-1/2}dxdy
\end{equation}
\begin{lemma}\label{lemma:derivativebound}
    Let $m\geq 1$ be a positive integer, then for $T\geq 10$ say,
    \begin{equation}\Bigr|\frac{\partial^m}{\partial y^m}\rho(r_T(z))\Bigr|\ll y^{-m}
    \end{equation}
    where $z=x+iy \in\mathbb{H} $ and the implicit constant depends only on $m$ and $\rho$.
\end{lemma}
\begin{proof}
    For sake of notation, we take $r^{(m)}(z)$ to denote the $m$-th derivative with respect to $y$. By the generalised chain rule 
    $$\frac{\partial^m}{\partial y^m}\rho(r_T(z))=\sum_{j=1}^{m}\sum_{\substack{a_1+...+a_j=m \\m\geq a_1\geq ... \geq a_j \geq 1 }}A_{a_1,...,a_j}\rho^{(m)}(r(z))\ r^{(a_1)}(z)...r^{(a_j)}(z)$$
    where $A_{a_1,...a_j}$ are some non-negative constants. We claim that for $a>0$
    $$r^{(a)}(z)=\frac{P_a(x,y)}{y^a((x^2+y^2+1)^2-4y^2)^{a-1/2}}$$
    with $P_a(z)$ a polynomial of total degree $\leq 4a-2$.
    As the base case
    $$r^{(1)}(z)=\frac{y^2-x^2-1}{y((x^2+y^2+1)^2-4y^2)^{1/2}}$$
    So suppose our claim holds for some $a$, then by the product rule
    \begin{align*}
            r^{(a+1)}(z)  = & -\frac{aP_a(x,y)}{y^{a+1}((x^2+y^2+1)^2-4y^2)^{a-1/2}}+\frac{P^{(1)}(x,y)}{y^{a}((x^2+y^2+1)^2-4y^2)^{a-1/2}} \\
            & \hspace{1cm} -\frac{2(2a-1)y(y^2+x^2-1)P_a(x,y)}{y^a((x^2+y^2+1)^2-4y^2)^{a+1/2}}
    \end{align*}
    Collecting terms, the denominator has the desired form, and the numerator is
    \begin{align*}
        -a((x^2+y^2+1)^2-4y^2)P_{a}(x,y)+y((x^2+y^2+1)^2-4y^2)& P_a^{(1)}(x,y) \\ 
        & -2(2a-1)y^2(x^2+y^2-1)P_a(x,y)
    \end{align*}
    whose degree satisfies the desired bound, thus the claim.\\
    Rearranging
    $$r^{(a)}(z)=\frac{1}{y^a}\frac{P_a(x,y)}{(x^2+y^2+1)^{2a-1}}\left(1-\frac{4y^2}{(x^2+y^2+1)^2}\right)^{1/2-a}$$
    By the support of the derivatives of $\rho$, we may assume $x+iy$ lies outside a small neighborhood of $i$, thus the third factor is $\ll_a 1$. For the second factor, if $|x|\geq y$ then $|P_a(x,y)|\ll_a |x|^{4a-2}$ and 
    $$\frac{|x|^{4a-2}}{(x^2+1)^{2a-1}}\leq 1$$
    A similar result holds for $|x| \leq y$, thus proving the lemma.
\end{proof}

\begin{proposition} \label{prop:polydecay}
    Let $\lambda_T(t):=\lambda_T(\phi_2)$ as below equation \eqref{eq:expression-lambda-T-principal}, with $\phi_2 \in \pi(\text{sgn}^\epsilon |\cdot|^{it},\text{sgn}^\epsilon|\cdot|^{-it})$. Fix $A\geq1$ be a positive integer, then
    \begin{equation}|\lambda_{T}(t)|\ll_A\frac{1}{(|t|+1)^A}M_A(T)
    \end{equation}
    for some function $M_A$ such that $M_A(T)\to 0$ as $T\to \infty$.
\end{proposition}
\begin{proof}
    Integrating by parts the integral expression \eqref{eq:eigenvalue} for $\lambda_{T}(t)$ we obtain
    $$\lambda_{T}(t)=\prod_{j=1}^{A}\frac{1}{s-1/2+j}\int_{0}^{\infty}\int_{-\infty}^{\infty}\frac{\partial^A}{\partial y^A}\frac{\rho(r_T(z))}{(x+i(1+y))^2}dx \ y^{it-1/2+A}dy$$
    where exchanging integral and the partial derivative is justified by smoothness and compact support of $f_\infty^T$.
    Therefore we have
    \begin{equation} \label{eq:function-tends-to-zero}
    \lambda_{T}(t)\ll\frac{1}{(|t|+1)^A}\int_{0}^{\infty}\Bigr|\int_{-\infty}^{\infty}\frac{\partial ^A}{\partial y^A}\frac{\rho(r_T(z))}{(x+i(1+y))^2}dx\Bigr|y^{\sigma-1/2+A}dy
    \end{equation}
    where $\sigma$ is the maximum $\text{Re}(it)$ for $t$ ranging over the spectral parameters occurring in the decomposition of $L_0^2([G])$ into cuspidal automorphic representations, and so can be taken as $\sigma=3/16$ (see e.g \cite[Theorem 4]{Deshouillers}). Take $M_A(T)$ to be the double integral in \eqref{eq:function-tends-to-zero}. We now show that $\lim_{T\to 0}M_A(T)=0$; this is accomplished by exchanging twice limit and integration via dominated convergence. By the generalised product rule
    \begin{equation}\label{eq:derivative-bound}
        \begin{split}
            \frac{\partial^A}{\partial y^A}\frac{\rho(r_T(z))}{(x+i(1+y))^2} & =\sum_{j=0}^{A}\binom{A}{j}\frac{\partial ^j}{\partial y^j}\rho(r_T(z))\frac{i^{A-j}}{(x+i(1+y))^{2+A-j}}\\
            & \ll_A y^{-A}\frac{1}{x^2+(1+y)^2}      
        \end{split}
    \end{equation}
    using Lemma \ref{lemma:derivativebound}. The modulus of the $y$-integrand is majorised (up to a constant dependent on $A$) by 
    $$\int_{-\infty}^{\infty}\frac{1}{x^2+(1+y)^2}dx\ y^{\sigma-1/2}=\frac{y^{\sigma-1/2}}{1+y}$$
    which is an integrable function on $(0,\infty)$, we can exchange the limit and the integral with respect to the variable $y$. It now suffices to show that
    $$\lim_{T\to \infty}\int_{-\infty}^{\infty}\frac{\partial ^A}{\partial y^A}\frac{\rho(r_T(z))}{(x+i(1+y))^2}dx=0$$
    Again the integrand is majorised (up to $A$-dependency) in absolute value by the same bound as in \eqref{eq:derivative-bound}.
    By construction $\rho(r_T(z))\to1$ pointwise and its derivatives $\rho^{(k)}(r_T(z))\to 0$ pointwise as $T\to \infty$ for $k \geq 1$. The result follows from noting that for $y > 0$
    $$\int_{-\infty}^{\infty}\frac{dx}{(x+i(1+y))^{2+A}}=0$$
    by applying Cauchy theorem to a semicircle in the upper half plane with radius approaching infinity.
\end{proof}
We remark that significantly better bounds than $3/16$ are currently known, but this weaker result is sufficient in our context. 

\subsection{Limiting argument for the spectral side}
We let again $\phi_2$ be the weight $2$ vector in some principal series $\pi(\text{sgn}^\epsilon|\cdot|^{it},\text{sgn}^\epsilon|\cdot|^{-it})$ for $\epsilon \in\{0,1\}$ and $t \in \mathbb{R}\cup(-\theta,\theta)i$ where we may take $\theta=3/16$ as in the proof of Proposition \ref{prop:polydecay}. Taking the evident weight two vector in the induced model and moving it to the Whittaker model using \cite[(3.10)]{Michel_2010} we obtain
\begin{equation}
    W_{\phi_2}(a(y))=\frac{(\text{sgn} \ y)^\epsilon}{\sqrt{\pi}}|y|^{1/2-s}\int_{-\infty}^{\infty}\frac{(x+i)^2}{(1+x^2)^{s+3/2}}e(yx)dx
\end{equation}
This integral is absolutely convergent only for $\text{Re}(s)>0$, but it admits a meromorphic continuation to the complex plane. We may integrate by parts once to obtain an integral expression that converges absolutely for $\text{Re}(s)>-1$, namely
\begin{align*}
    W_{\phi_2}(a(y))=& 2\frac{(\text{sgn} \ y)^\epsilon}{\sqrt{\pi}}|y|^{1/2-s}\frac{1}{2\pi i y} \\ & \qquad \times \left[\int_{-\infty}^{\infty}\frac{x+i}{(1+x^2)^{s+3/2}}e(yx)dx  
     -(s+3/2)\int_{-\infty}^{\infty}\frac{x(x+i)^2}{(1+x^2)^{s+5/2}}e(yx)dx \right]
\end{align*}
A straightforward computation then shows that
\begin{equation}\label{eq:whittaker-bound}
    W_{\phi_2}(a(y))\ll |y|^{-1/2+\theta}(1+2|t|)
\end{equation}
where the implicit constant is absolute. We may thus obtain the following estimates:

\begin{proposition}\label{prop:whittaker-growth-2}
    Fix $T \geq 1$ and $m \in \mathbb{Z}$. Then for $K \in \mathbb{R}_+$ sufficiently large
    \begin{equation}
        \sum_{\substack{\pi \in \mathcal{F}_{0}(T,f)\\ |t_\pi|\leq K}}\sum_{\phi \in \mathcal{B}(\pi)}|W_\phi(a(-m))|^2 \ll K^{6}
    \end{equation}
    where $t_\pi$ is the spectral parameter of $\pi$, and similarly
    \begin{equation}
        \sum_{\substack{\chi \in \mathcal{F}_E}(T,f)}\sum_{\phi \in \mathcal{B}(\chi,\chi^{-1})}\int_{-K}^K |W_{E(\phi_{it})}(a(-m))|^2 dt \ll K^{3+\epsilon}
    \end{equation}
    In both cases, the implicit constants depend on $m,N, \theta$ and $\epsilon$. 
\end{proposition}
\begin{proof}
    Using the inequality \eqref{eq:whittaker-bound}, we may bound
    $$ \sum_{\substack{\pi \in \mathcal{F}_{0}(T,f)\\ |t_\pi|\leq K}}\sum_{\phi \in \mathcal{B}(\pi)}|W_\phi(a(-m))|^2 \ll_m\sum _{\substack{\pi \in \mathcal{F}_{0}(T,f)\\ |t_\pi|\leq K}}(1+|t_\pi|)^2\sum_{\varphi \in \mathcal{B}(\pi)}|a_{u_\varphi}(m)|^2$$
    hence the result follows from Lemma \ref{lemma:maass-growth}. The bound for the Eisenstein series is similarly proven: the family $\mathcal{F}_E(f)$, and by extension $\mathcal{F}_E(T,f)$, are bounded in terms of $N$. Moreover, using e.g. (\cite{KL-kuznetsov}, \S 7.4) we have the (very weak) bound
    \begin{equation*}
        a_{u_{E(\phi_{it})}}(m)\ll_{m,N,\epsilon} \log(3+2|t|)^7
    \end{equation*}
    so that $|W_{E(\phi_{it})}(a(-m))|\ll_{m,N,\epsilon} (1+|t|)^{1+\epsilon}$.
\end{proof}

We are now ready to prove the main result of this section, in which we recover the spectral side of the equation for $\kappa=2$.
\begin{proposition}\label{prop:spectral-limit}
    As $T \to \infty$, the behavior of the spectral side of is as follows
    \begin{equation}\label{eq:truncated-cusp}
        \lim_{T \to \infty} I_{\operatorname{cusp}}(T)=(4\pi)^2\sqrt{m_1 m_2}e^{-2\pi(m_1+m_2)} \sum_{\pi \in \mathcal{F}_2(f)}\sum_{\varphi \in \mathcal{B}(\pi)} a_{u_\varphi}(m_1)\overline{a_{u_\varphi}(m_2)}
    \end{equation}
    For the residual and continuous contribution instead
    \begin{equation}
        \lim_{T \to \infty} I_{\operatorname{res}}(T)=\lim_{T \to \infty}I_{\operatorname{cts}}(T)=0
    \end{equation}
\end{proposition}
\begin{proof}
    For all $T\geq 1$, $I_\text{res}(T)=0$ since the residual space is spanned by $K_\infty$-invariant vectors. We might split $I_\text{cusp}(T)$ depending on the infinity type, so
    \begin{equation}
        \begin{split}
        I_\text{cusp}(T)  = &(4\pi)^2\sqrt{m_1m_2}e^{-2\pi(m_1+m_2)}\sum_{\pi \in \mathcal{F}_2(T,f)}\sum_{\varphi \in \mathcal{B}(\pi)}\lambda_{T}(\varphi)a_{u_\varphi}(m_1)\overline{a_{u_\varphi}(m_2)}\\ & + \sum_{\pi \in \mathcal{F}_0(T,f)}\lambda_{T}(t_\pi)\sum_{\varphi \in \mathcal{B}(\pi)}W_{\varphi}(a(-m_1))\overline{W_\varphi(a(-m_2))}
        \end{split}
    \end{equation}
    Choosing $A$ sufficiently large in Proposition \ref{prop:polydecay}, using Proposition \ref{prop:whittaker-growth-2}, the second line in vanishes as $T \to \infty$. The first line tends to the right hand side of \eqref{eq:truncated-cusp} by Proposition \ref{prop:discrete-limit}. Similarly 
    \begin{equation}
        I_\text{cts}(T)=\sum_{\chi \in \mathcal{F}_E(T,f)}\sum_{\phi \in \mathcal{B}(\chi,\chi^{-1})}\int_{-\infty}^{\infty}\lambda_T(t)W_{E(\phi_{it})}(-m_1)\overline{W_{E(\phi_{it})}(-m_2)}dt
    \end{equation}
    and given that the two outer sums are finite, by taking $k$ sufficiently large we see that this term also vanishes as $T \to \infty$.
\end{proof}

\subsection{Limiting argument for the geometric side}
We now apply a similar limiting argument as $T \to \infty$ to recover the geometric side. Again we consider $I_\delta(f_\mathbb{A}^T)$ (see \eqref{eq:orbit-geom-integral}) where $\delta$ is a first or second cell term given in \eqref{def:cells}. As the truncation only affects the archimedean component, the discussion of $I_\delta(f)$ is exactly the same as in Section 3. \bigbreak 
\noindent \emph{5.3.1 First cell term}: Unlike the BK formula for opposite signs, the diagonal contribution on the geometric side does not always vanish. We will see that although $f_\infty$ is not integrable, upon restricting the integration to the upper triangular matrices with fixed diagonal entries we get the following
\begin{lemma}
Let $m_1, m_2 \in \mathbb{Q}_+$, and let $\delta = \left(\begin{smallmatrix}
    m_1/m_2 & 0 \\ 0 & 1
\end{smallmatrix}\right) \in G(\mathbb{Q})$,  then 
\begin{equation}
    \lim_{T \to \infty}I_\delta(f_\infty^T)= 4\pi(m_1m_2)e^{-2\pi(m_1+m_2)}
\end{equation}
\end{lemma}
\begin{proof}
    Consider the following functions
    $$F^T(t)=f_\infty^T\left(\begin{pmatrix}
        m_1 & t \\ 0 & m_2
    \end{pmatrix} \right)\theta_\infty(-t) \hspace{0.5cm} \text{and} \hspace{0.5cm} G(t)=f_\infty \left(\begin{pmatrix}
        m_1 & t \\ 0 & m_2
    \end{pmatrix} \right)\theta_\infty(-t)$$
    Since $|F^T(t)|\leq |G(t)|$ for all $t \in \mathbb{R}, T \geq 0$, and $G(t)\in L^1(\mathbb{R})$, by dominated convergence theorem
\begin{equation}
    \lim_{T \to \infty} I_{\delta}(f_\infty^T) = \int_{\mathbb R}f_{\infty}\left(\begin{pmatrix}
        m_1 & t \\ 0 & m_2 
    \end{pmatrix} \right)\theta_{\infty}(-t)dt= -\frac{m_1m_2}{\pi}\int_{-\infty}^{\infty}\frac{e^{2\pi it}}{(-t+i(m_1+m_2))^2}dt
\end{equation}
We evaluate the integral on the right of by contour integration on a semicircle in the upper half plane. Since the only pole occurs at $i(m_1+m_2)$, as the radius tends to infinity the residue theorem gives
$$I_\delta(f_\infty^T) =-\frac{m_1m_2}{\pi}\cdot 2\pi i \frac{d}{dt}e^{2\pi it} \Bigr|_{t=(m_1+m_2)i} = 4\pi(m_1m_2)e^{-2\pi (m_1+m_2)}$$
which gives the desired result.
\end{proof}

\noindent \emph{5.3.2 Second cell terms:}
For $\delta = \left(\begin{smallmatrix}
        0 & -\mu \\ 1 & 0
    \end{smallmatrix}\right)$ with $\mu \in \mathbb{Q_+}$ we have the following bound on the archimedean integral $I_\delta(f_\infty^T)$:
\begin{lemma}\label{lemma:second-cell-archim-bound}
Let $\mu > 0$, then
\begin{equation}
    \iint_{\mathbb{R}\times\mathbb{R}}\Bigr|f_\infty\left(\begin{pmatrix}
    -t_1 & -\mu-t_1t_2 \\ 1 & t_2
\end{pmatrix}\right)\Bigr|dt_1dt_2 \ll \mu
\end{equation}
\end{lemma}
\begin{proof}
    By the explicit expression for $f_\infty$ in \eqref{def:matrix-coeff}, the integral is
    $$\frac{\mu}{\pi}\iint_{\mathbb{R}\times \mathbb{R}}\frac{dt_2dt_1}{|t_2(t_1+i)+\mu-i(t_1+i)|^2} = \frac{\mu}{\pi} \iint_{\mathbb{R}\times \mathbb{R}}\frac{dt_2 dt_1}{|(t_1+i)(t_2-(i-\frac{\mu}{t_1+i}))|^2}$$
    The squared modulus of the second factor in the denominator  is
    $$\left(t_2+\frac{\mu t_1}{t_{1}^2+1}\right)^2+\left(\frac{\mu}{t_{1}^2+1}+1\right)^2$$
    thus upon making the change of variables $t_2\mapsto t_2-\mu t_1/(t_{1}^{2}+1)$, the inner integral becomes
    $$\int_{-\infty}^{\infty}\frac{dt_2}{t_2^2+(\frac{\mu}{t_1^2+1}+1)^2} = \frac{\pi}{\left(\frac{\mu}{t_1^2+1}+1\right)}\leq \pi$$
    It is then clear that the outer integral is bounded by $\pi^2$ and the result follows.
\end{proof}

\begin{proposition}\label{prop:geometric-limit}
    Let $m_1,m_2 \in \mathbb{Q}^\times$, then 
    \begin{equation}
        \begin{split}
        \lim_{T \to \infty}I_{\operatorname{geom}}(T)= & 4\pi(m_1m_2)e^{-2\pi(m_1+m_2)}\delta_{m_1=m_2}\delta\\ & -8\pi^2(m_1m_2)^{1/2}\sum_{c \in \mathcal{C}(\mathcal{F})}\frac{H(m_1,m_2,c)}{c}J_1\left(\frac{4\pi\sqrt{m_1m_2}}{c}\right)
        \end{split}
    \end{equation}
    if $m_1,m_2 \in \mathbb{N}$, and vanishes otherwise.
\end{proposition}
\begin{proof}
For the computation of the non-archimedean integral $I_\delta(f)$ where $\delta$ is the cell of first kind, we refer the reader to \cite[\S 4.2]{generalised} . For the archimedean integral for a cell of the second kind, we have
\begin{equation}
    \begin{split}
\iint_{\mathbb{R}\times \mathbb{R}}f_\infty(\begin{pmatrix}
    -t_1 & -\mu-t_1t_2 \\ 1 & t_2
\end{pmatrix})\theta_\infty(m_1t_1-m_2t_2)dt_1dt_2 \\ 
= -8\pi^2\sqrt{\mu m_1m_2}e^{-2\pi(m_1+m_2)}J_1(4\pi\sqrt{\mu m_1m_2})
    \end{split}
    \end{equation}
This was proven in \cite[Proposition 3.6]{KL-petersson} for $\kappa > 2$ but holds as well for $\kappa=2$ with no changes. For the non-archimedean integral of a cell of the second kind of the form $\begin{psmallmatrix}
    & -\mu \\ 1 &
\end{psmallmatrix}$ for $\mu=c^{-2}$ (see \eqref{eq:parametrisation}), we have $I_\delta(f)=H(m_1,m_2,c)$. Therefore the previous results are sufficient to prove the proposition, provided that we can exchange the order of the limit and the summation / double integration. To this purpose, we introduce 
    $$ G_\mu(t_1,t_2)=f_\infty (\begin{pmatrix}
        -t_1 & -\mu-t_1t_2 \\ 1 & t_2
    \end{pmatrix} )$$
    By the dominated convergence theorem, exchanging sum and integration is justified provided that 
    \begin{equation}
        {\sum_{c \in \mathcal{C}(\mathcal{F})}}|H(m_1,m_2;c)|\iint_{\mathbb{R}\times \mathbb{R}}|G_{1/c^2}(t_1,t_2)|dt_1dt_2
    \end{equation}
    converges absolutely. Using the factorisation property \eqref{eq:kloosterman-factorisation} of the generalised Kloosterman sums and Lemma \ref{lemma:second-cell-archim-bound}, the sum is bounded above (up to a constant) by
    \begin{equation}
        \sum_{\substack{c_N | N^\infty \\ k(\mathcal{F})|c_N}}\sum_{(c_0,N)=1}\frac{1}{c_0^2c_N^2}|S(\overline{c_N}m_1,\overline{c_N}m_2,c_0)H(m_1\overline{c_0},m_2\overline{c_0},c_N)|
    \end{equation}
Using the Weyl bound on the classical Kloosterman sums, and the trivial bound from Proposition \ref{prop:kloosterman-properties}, this is 
\begin{equation}
    \ll \sum_{\substack{c_N | N^\infty \\ k(\mathcal{F})|c_N}}\frac{1}{c_N}\sum_{(c_0,N)=1}\frac{1}{c_0^{3/2-\epsilon}}< \infty
\end{equation}
where the implied constant depends on $N, m_1,m_2,\epsilon, f(1)$ and $\kappa(\mathcal{F})$. The vanishing if $m_1,m_2 \not \in \mathbb{N}$ follows from Proposition \ref{prop:kloosterman-properties} \emph{(4)}.
    
\end{proof}

We conclude remarking that combining Propositions \ref{prop:spectral-limit} and \ref{prop:geometric-limit}, we obtain Theorem \ref{thm:p2-formula} upon dividing both sides by $16\pi^2(m_1m_2)^{1/2}e^{-2\pi(m_1+m_2)}$ and rewriting the spectral term in terms of Hecke eigenvalues using \eqref{def:discrete-weights}.

\pagebreak
\printbibliography
\end{document}